\documentclass{amsart}
\usepackage{amsmath,amsthm,latexsym,amssymb,hyperref,amsfonts}
\usepackage{stmaryrd}
\usepackage{eucal}
\usepackage{mathrsfs}
\usepackage{comment}
\usepackage[all]{xy}
\usepackage[usenames,dvipsnames]{xcolor}
\usepackage{pdflscape}
\allowdisplaybreaks
\newtheorem*{mainthm}{Acyclicity Theorem}

\newtheorem{thm}{Theorem}[subsection]
\newtheorem{lem}[thm]{Lemma}
\newtheorem{prop}[thm]{Proposition}
\newtheorem{cor}[thm]{Corollary}

\theoremstyle{definition}
\newtheorem{defn}[thm]{Definition}

\theoremstyle{remark}
\newtheorem{rmk}[thm]{Remark}
\newtheorem{ex}[thm]{Example}
\newtheorem{exs}[thm]{Examples}

\newtheorem{notn}[thm]{Notation}

\newtheorem{terminology}[thm]{Terminology}
\newtheorem{conv}[thm]{Convention}

\makeatletter
\let\c@equation\c@thm
\makeatother
\numberwithin{equation}{subsection}

\DeclareMathAlphabet{\mathbbe}{U}{bbold}{m}{n}

\newcommand\Op{\mathcal}

\newdir{ >}{{}*!/-10pt/@{>}}

\newcommand{\ob}{\operatorname{Ob}}
\newcommand{\mor}{\operatorname{Mor}}

\newcommand{\id}{\mathrm{id}}
\newcommand{\Ho}{\operatorname{Ho}}

\newcommand\op{\mathrm{op}}

\newcommand\coker{\operatorname{coker}}

\newcommand\lan{\operatorname{Lan}}

\renewcommand\hom{\operatorname{Hom}}
\newcommand{\colim}{\operatorname{colim}}
\newcommand{\Iso}{\mathrm{Iso}}
\renewcommand\Bar{\operatorname{Bar}}
\newcommand{\Cyl}{\operatorname{Cyl}}

\newcommand{\lra}{\longrightarrow}

\newcommand \ve{\varepsilon}
\newcommand\Om {\Omega}
\newcommand{\si}{s^{-1}}
\newcommand{\del}{\partial}

\newcommand\cat[1]{\mathsf{#1}}
\newcommand{\Cat}{\mathsf{Cat}\,}
\newcommand{\Set}{\mathsf{Set}}
\newcommand{\Ch}{\mathsf{Ch}}
\newcommand{\sSet}{\mathsf{sSet}}
\newcommand{\Mod}{\mathsf{Mod}}
\newcommand{\CAT}{\mathsf{CAT}\,}
\newcommand{\Alg}{\mathsf{Alg}}
\newcommand{\Coring}{\mathsf{Coring}}
\newcommand{\Coalg}{\mathsf{Coalg}}

\newcommand{\Bialg}{\mathsf{Bialg}}
\newcommand{\SpS}{\mathsf{Sp}^{\Sigma}}

\newcommand{\A}{\mathsf A}
\newcommand{\B}{\mathsf B}
\newcommand{\C}{\mathsf C}
\newcommand{\D}{\mathsf D}
\newcommand\R{\mathsf R}
\newcommand\V{\mathsf V}
\newcommand{\K}{\mathsf K}
\newcommand{\M}{\mathsf M}
\newcommand{\N}{\mathsf N}
\renewcommand\P{\mathsf P}

\renewcommand\R{\mathsf R}

\newcommand\cL{\mathcal L}
\newcommand\X{\mathcal X}
\newcommand\I{\mathcal I}
\newcommand\J{\mathcal J}
\newcommand\cR{\mathcal R}

\newcommand{\cP}{\mathcal{P}}
\newcommand{\cQ}{\mathcal{Q}}

\newcommand\TT{\mathbb T}
\newcommand\KK{\mathbb K}
\newcommand\CC{\mathbb{C}}
\newcommand\FF{\mathbb{F}}

\newcommand{\WE}{\mathcal{W}}
\newcommand{\Fib}{\mathcal{F}}
\newcommand{\Cof}{\mathcal{C}}

\newcommand\uI {\mathbbe{1}}

\newcommand\adjunct[4]{\xymatrix@C=3pc@R=4pc{#1\ar @<1.25ex>[rr]^{#3}&\perp&#2\ar @<1.25ex>[ll]^{#4}}}

\newcommand{\pbtimes}[1]{\mathbin{\mathop{\times}\limits_{#1}}}

\renewcommand{\2}{\mathbf{2}} 
\newcommand{\3}{\mathbf{3}}
\newcommand{\LL}{\mathbb{L}}
\newcommand{\RR}{\mathbb{R}}
\newcommand{\UU}{\mathbb{U}}
\renewcommand{\AA}{\mathbb{A}}

\newcommand{\rlp}[1]{{#1}^\boxslash}
\newcommand{\llp}[1]{{}^\boxslash\!{#1}}

\DeclareMathOperator{\sm}{\wedge}

\newcommand{\coalg}[1]{\mathsf{Coalg}_{#1}}
\newcommand{\alg}[1]{\mathsf{Alg}^{#1}}

\newcommand{\CCoalg}[1]{\mathbb{C}\mathsf{oalg}_{#1}}
\newcommand{\AAlg}[1]{\mathbb{A}\mathsf{lg}^{#1}}

\newcommand{\Sq}[1]{\mathbb{S}\mathsf{q}(#1)}

\begin{document}
\title{A necessary and sufficient condition for induced model structures}
\author [Hess]{Kathryn Hess}
\author[K\c{e}dziorek]{Magdalena K\c{e}dziorek}
\author[Riehl]{Emily Riehl}
\author[Shipley]{Brooke Shipley}

\address{UPHESS BMI FSV \\
    \'Ecole Polytechnique F\'ed\'erale de Lausanne \\
    CH-1015 Lausanne \\
    Switzerland}
    \email{kathryn.hess@epfl.ch}

\address{UPHESS BMI FSV \\
    \'Ecole Polytechnique F\'ed\'erale de Lausanne \\
    CH-1015 Lausanne \\
    Switzerland}
\email{magdalena.kedziorek@epfl.ch}

\address{Department of Mathematics\\Johns Hopkins University\\3400 N.~Charles Street\\Baltimore, MD 21218\\ USA}
\email{eriehl@math.jhu.edu}

\address{Department of Mathematics, Statistics, and Computer Science, University of Illinois at
Chicago, 508 SEO m/c 249,
851 S. Morgan Street,
Chicago, IL, 60607-7045, USA}
    \email{bshipley@math.uic.edu}

\date {\today }


 \begin{abstract} 
{A common technique for producing a new model category structure is to lift the fibrations and weak equivalences of an existing model structure along a right adjoint. Formally dual but technically much harder is to lift the cofibrations and weak equivalences along a left adjoint. For either technique to define a valid model category, there is a well-known necessary ``acyclicity'' condition. We show that for a broad class of \emph{accessible model structures} --- a generalization introduced here of the well-known \emph{combinatorial model structures}  --- this necessary condition is also sufficient in both the right-induced and left-induced contexts, and the resulting model category is again accessible. We develop new and old techniques for proving the acyclity condition and apply these observations to construct several new model structures, in particular on categories of differential graded bialgebras, of differential graded comodule algebras, and of comodules over corings in both the differential graded and the spectral setting. We observe moreover that (generalized) Reedy model category structures can also be understood as model categories of ``bialgebras'' in the sense considered here.}
 \end{abstract}
\maketitle
\tableofcontents

\section{Introduction}

A \emph{model category} or, more precisely,  a \emph{model structure} on a category is an abstract framework for homotopy theory. A model category  $(\M, \Fib, \Cof, \WE)$ consists of a bicomplete category $\M$ and a class $\WE$ of \emph{weak equivalences} satisfying the 2-of-3 property, together with a pair of weak factorization systems $(\Cof \cap \WE, \Fib)$ and $(\Cof, \Fib \cap \WE)$ formed by the \emph{acyclic cofibrations} and \emph{fibrations} and by the \emph{cofibrations} and \emph{acyclic cofibrations}, respectively. These axioms allow one to define an abstract notion of homotopy equivalence between objects of $\M$, coinciding with the given weak equivalences on a certain subcategory of objects, and to understand when ordinary limit and colimit constructions are homotopy limits and colimits, i.e., are weak-equivalence invariant.

Given a bicomplete category $\C$ and a pair of adjoint functors
$$\adjunct{\M}{\C}{L}{R}$$ 
there are well known conditions under which there is a model structure on $\C$, which we call the \emph{right-induced model structure}, with $R^{-1}\WE$, $R^{-1}\Fib$ as weak equivalences and fibrations, respectively. The first conditions are set-theoretic and guarantee that the required factorizations can be constructed on $\C$. The final ``acyclicity condition'' implies that the lifted fibrations and weak equivalences are compatible with the class of cofibrations they determine; see \cite[Theorems 11.3.1 and 11.3.2]{hirschhorn}.

The dual case, where we are given a pair of adjoint functors $$\adjunct{\C}{\M}{L}{R}$$ and desire a \emph{left-induced model structure} on $\C$, with $L^{-1}\WE$, $L^{-1}\Cof$ as weak equivalences and cofibrations, is technically much more difficult. The ``acyclicity condition,'' guaranteeing the compatibility of the lifted cofibrations and weak equivalences with fibrations they determine, is formally dual, but the set-theoretic issues are much more complicated. A recent breakthrough result of Makkai-Rosicky \cite{makkai-rosicky}, applied in this context in  \cite{six-author}, describes how the set-theoretic obstacles can be overcome.

As formulated in \cite{six-author}, a common hypothesis handles the set-theoretic issues in both of the situations above and guarantees that the necessary factorizations can always be constructed. Specifically, it suffices to assume that the model structure on $\M$ is cofibrantly generated and that both categories $\M$ and $\C$ are locally presentable; a model structure is \emph{combinatorial} just when it is cofibrantly generated and its underlying category is locally presentable. Locally presentable categories ``permit the small object argument'' for any set of maps.

In this paper, we retain the hypothesis that all categories under consideration are locally presentable but relax the hypothesis that model structures are cofibrantly generated. Our primary motivation is to obtain a general theory that includes the Hurewicz model structure on $\Ch_R$, the category of unbounded chain complexes of modules over a commutative ring $R$, in which the weak equivalences are chain homotopy equivalences. Christensen and Hovey show that in the case where $R=\mathbb{Z}$, the Hurewicz model structure is not cofibrantly generated \cite{christensen-hovey}; thus this model category is not combinatorial. It is, however, cofibrantly generated in a more general enriched sense, which the authors of \cite{barthel-may-riehl} use to right-induce model structures that are again enriched cofibrantly generated. However, the techniques in \cite{barthel-may-riehl}, making use of an enriched version of Quillen's small object argument, {cannot be applied to left-induce the enriched cofibrantly generated model structures under consideration there and thus do not suffice for our present purposes.}

In this paper, we introduce \emph{accessible model structures}, a generalization of the familiar combinatorial model structures on locally presentable categories for which left- and right-induction of model-theoretic functorial factorizations is nonetheless always possible. Thus, left- or right-induced model structures from accessible model structures exist if and only if a simple ``acyclicity'' condition is satisfied, as we describe in Section \ref{sec:acyclicity}. All combinatorial model structures are accessible. Theorem \ref{thm:acc-awfs} demonstrates that the enriched cofibrantly generated model structures under consideration in \cite{barthel-may-riehl} are also accessible. In particular, this includes the Hurewicz model structure on $\Ch_R$ for any commutative ring $R$. Left- and right-induced model structures of accessible model structures are again accessible. Thus the processes of left- and right-induction can be iterated as needed.

This state of affairs is enabled by some rather sophisticated categorical machinery due to Bourke and Garner \cite{bourke-garner}, the details of which are described in Section \ref{sec:induced_awfs}, but whose outline we present now. Grandis and Tholen \cite{grandis-tholen} introduced what are now called \emph{algebraic weak factorization systems}, a more structured variant of the weak factorization systems used to define a model category that have superior categorical properties. Loosely, an algebraic weak factorization system is an ordinary weak factorization system together with a well-behaved functorial factorization; an algebraic weak factorization system is \emph{accessible} if its functorial factorization preserves (sufficiently large) filtered colimits. Garner \cite{garner-understanding} introduced an improved version of Quillen's small object argument that produces an accessible algebraic weak factorization system from any set of generating arrows. In particular, any combinatorial model category can be equipped with a pair of accessible algebraic weak factorization systems. An \emph{accessible model structure} is a model structure on a locally presentable category equipped with a pair of accessible algebraic weak factorization systems. For this broad class of model structures, we have the following powerful existence result for left- and right-induction, which is the most general result of this kind known to date.  The existence criteria are expressed in terms of distinguished classes of morphisms defined as follows.

\begin{notn} Let $f$ and $g$ be morphisms in a category $\C$. If for every commutative diagram in $\C$
$$\xymatrix{ \cdot \ar[d]_f \ar[r]^{{a}} & \cdot \ar[d]^{g} \\ \cdot \ar[r]_{{b}} \ar@{-->}[ur]_{c} & \cdot}$$
the dotted lift $c$ exists, i.e., $gc=b$ and $cf=a$, then we write $f\boxslash g$.

If $\X$ is a class of morphisms in a category $\C$, then
$$\llp\,{\X}=\{ f \in \mor C\mid f\boxslash x\quad \forall x\in \X\},$$
and
$$\rlp{\X}=\{ f \in \mor C\mid x\boxslash f\quad \forall x\in \X\}.$$
\end{notn}

\begin{mainthm}[Corollary \ref{cor:acyclicity-reduction}]
Suppose $(\M, \Fib, \Cof,\WE)$ is an accessible model category, $\C$ and $\K$ are locally presentable, and there exist adjunctions 
\[ \xymatrix@C=4pc@R=4pc{ \K \ar@<1ex>[r]^V \ar@{}[r]|\perp & \M \ar@<1ex>[l]^R \ar@<1ex>[r]^L\ar@{}[r]|\perp & \C. \ar@<1ex>[l]^U}\]  
\begin{enumerate}
\item The right-induced model structure exists on $\C$ if and only if \[\llp\,{U^{-1}\Fib} \subset U^{-1}\WE.\]
\item The left-induced model structure exists on $\K$ if and only if \[\rlp{(V^{-1}\Cof)} \subset V^{-1}\WE.\]
\end{enumerate}
{When these conditions are satisfied, the induced model structures on $\C$ and on $\K$ are again accessible.}
\end{mainthm}

In an accessible model category, the majority of the components of the underlying model structure have been ``described algebraically'' and thus become easy to transport.\footnote{There is a 2-monad on the 2-category of locally presentable categories and accessible functors whose algebras are locally presentable categories equipped with a pair of accessible algebraic weak factorization systems with a comparison map between them.} By \cite{BKP}, the 2-category of such structures admits certain 2-dimensional limits, which are used to facilitate left- and right-induction. The pair of nested functorial factorizations encode classes of ``acyclic cofibrations,'' ``cofibrations,'' ``acyclic fibrations,'' and ``fibrations.'' The remaining non-algebraic datum is the compatibility of the classes of acyclic cofibrations and acyclic fibrations with the weak equivalences; this is the \emph{acyclicity condition} {appearing in (1) and (2) above.}

The acyclicity condition is generally not easy  to check, but we provide several methods to do so in Section \ref{subsection:acyclicity}.  
We apply these methods to produce new model structures on  categories of {comodules over corings in both spectral and differential graded frameworks;} differential graded associative bialgebras and $H$-comodule algebras; and  Reedy categories.

This paper is structured as follows. In Section \ref{sec:acyclicity} we describe various methods of checking the acyclicity conditions, enabling us to apply the Acyclicity Theorem to produce the new model structures mentioned above. The proof of the Acyclicity Theorem is given in Section \ref{sec:induced_awfs}, where we  also outline the relevant work of Bourke and Garner \cite{bourke-garner} upon which our result is based. {To illustrate the use of this theorem, we prove that the category $\D$-shaped diagrams valued in any accessible model category admits both projective and injective model structures, generalizing the corresponding result for combinatorial model categories.} In Section \ref{sec:enriched-cof-gen}, we prove that a large class of enriched cofibrantly generated model categories, including in particular the Hurewicz model structure on $\Ch_R$, are accessible model categories, to which the Acyclicity Theorem applies. The rest of the paper presents numerous examples, starting with the spectral examples in Section \ref{sec:spectra_ex} and followed by DG examples in Section \ref{section:dg_examples}. {With one small exception,} model structures defined in these sections are new. In the last section we explain how to understand (generalized) Reedy model category structure using left and right induction and Theorem \ref{thm:square}. In the addendum we observe that left induction is more common than one may think, occuring even in the well-known case of the adjunction between simplicial sets and topological spaces (with either of two familiar model structures).

\begin{notn}
Throughout this article, weak equivalences are denoted $\xrightarrow \sim$, while cofibrations are denoted $\rightarrowtail$.
\end{notn}

{\begin{conv} Unless stated explicitly otherwise, all monoidal categories in this article are assumed to be symmetric.  Moreover, when we refer to a \emph{monoidal model category}, we mean a \emph{closed, symmetric monoidal model category}, in the sense of \cite{schwede-shipley}.
\end{conv}}

 \subsection{Acknowledgments} Our collaboration began at the first \emph{Women in Topology} workshop, which was held in 2013 at the Banff International Research Station (BIRS), sponsored by both BIRS and the Clay Mathematics Foundation, to whom we express our gratitude for their support.  
 
 Part of this article is also based upon work supported by the National Science Foundation under Grant No.~0932078000 while the authors were in residence at the Mathematical Sciences Research Institute in Berkeley, California, during the Spring 2014 semester.  We would also like to thank the University of Illinois at Chicago and the EPFL for their hospitality during research visits enabling us to complete the research presented in this article. The third author was supported by NSF grants DMS-1103790 and DMS-1509016 and the fourth author by NSF grants DMS-1104396 and DMS-1406468.  

{The authors thank Gabriel Drummond-Cole for pointing out an error in the original formulation of Proposition 6.2.1 and the referee for a careful  and helpful report.}

\section{Accessible model categories, acyclicity, and induced model structures}\label{sec:acyclicity}

 {In this section, we review the acyclicity condition and explain why it suffices to guarantee the compatibility of the cofibrations, fibrations, and weak equivalences in an induced model structure. We then present several techniques for proving the acyclicity condition, which we apply throughout the second half of this paper.}

\subsection{Induced model structures and acyclicity}

\begin{defn} A \emph{weak factorization system} on a category $\C$ consists of a pair $(\cL,\cR)$ of classes of maps so that the following conditions hold.
\begin{itemize}
\item Any morphism in $\C$ can be factored as a morphism in $\cL$ followed by a morphism in $\cR$.
\item $\cL = \llp{\cR}$ and $\cR = \rlp{\cL}$, i.e., any commutative square
\[ \xymatrix{ \bullet \ar[r] \ar[d]_{\cL \ni \ell} & \bullet \ar[d]^{r \in \cR} \\ \bullet \ar[r] \ar@{-->}[ur] & \bullet}\] whose left-hand vertical morphism lies in $\cL$ and whose right-hand vertical morphism lies in $\cR$ admits a lift, and moreover $\cL$ contains every morphism with the left lifting property with respect to each $r \in \cR$, and $\cR$ contains every morphism with the right lifting property with respect to each $\ell \in \cL$.
\end{itemize}
\end{defn}

For example, if $(M,\Fib,\Cof,\WE)$ is a model category, then $(\Cof \cap \WE, \Fib)$ and $(\Cof, \Fib \cap \WE)$ are weak factorization systems. The converse also holds, under one additional condition.

\begin{prop}[{\cite[7.8]{joyal-tierney}}]\label{prop:model-via-wfs} If $\M$ is a bicomplete category, and $\Fib, \Cof, \WE$ are classes of morphisms so that 
\begin{itemize}
\item $\WE$ satisfies the 2-of-3 property, and 
\item $(\Cof \cap \WE, \Fib)$ and $(\Cof, \Fib \cap \WE)$ are weak factorization systems,
\end{itemize}
then $(\M,\Fib,\Cof,\WE)$ defines a model category.
\end{prop}

\begin{defn}
Let $(\M,\Fib,\Cof,\WE)$ be a model category, and consider a pair of adjunctions
\[ \xymatrix@C=4pc@R=4pc{ \K \ar@<1ex>[r]^V \ar@{}[r]|\perp & \M \ar@<1ex>[l]^R \ar@<1ex>[r]^L\ar@{}[r]|\perp & \C \ar@<1ex>[l]^U}\]
where the categories $\C$ and $\K$ are bicomplete.
If they exist:
\begin{itemize}
\item the \emph{right-induced model structure} on $\C$ is given by \[\big(\C,U^{-1}\Fib, \llp\,{U^{-1}(\Fib \cap\WE)}, U^{-1}\WE\big),\] and
\item the \emph{left-induced model structure} on $\K$ is given by \[ \big(\K, \rlp{(V^{-1}(\Cof\cap\WE))}, V^{-1}\Cof, V^{-1}\WE\big).\]
\end{itemize}
\end{defn}

{In other words, in the right-induced model structure, fibrations and weak equivalences are created by $U$, while in the left-induced model structure, cofibrations and weak equivalences are created by $V$.}

When they exist,  the left- and right-induced model structures can be characterized easily as follows.
\begin{itemize}
\item If the right-induced model structure exists on $\C$, then both of its weak factorization systems are right-induced from the weak factorization systems on $\M$, {i.e., the right classes are created by $U$.}
\item If the left-induced model structure exists on $\K$, then both of its weak factorization systems are left-induced from the weak factorization systems on $\M$, {i.e., left classes are created by $V$.}
\end{itemize}

In Section \ref{sec:induced_awfs}, we show that if $(\M, \Fib, \Cof,\WE)$ is an accessible model category, and $\C$ and $\K$ are locally presentable, then the right-induced weak factorization systems, whose right classes are $U^{-1}\Fib$ and $U^{-1}(\Fib \cap\WE)$, exist on $\C$, and the left-induced weak factorization systems, whose left classes are $V^{-1}(\Cof\cap\WE)$ and $V^{-1}\Cof$, exist on $\K$. The following result explains how the existence of right- and left-induced model structures reduces to the acyclicity condition.

\begin{prop}\label{prop:acyclicity-reduction} Suppose $(\M, \Fib, \Cof,\WE)$ is a model category, $\C$ and $\K$ are bicomplete categories, and there exist adjunctions 
\[ \xymatrix@C=4pc@R=4pc{ \K \ar@<1ex>[r]^V \ar@{}[r]|\perp & \M \ar@<1ex>[l]^R \ar@<1ex>[r]^L\ar@{}[r]|\perp & \C \ar@<1ex>[l]^U}\] so that the right-induced weak factorization systems exists on $\C$, and the left-induced weak factorization systems exists on $\K$.  It follows that
\begin{enumerate}
\item the right-induced model structure exists on $\C$ if and only if \[\llp\,{U^{-1}\Fib} \subset U^{-1}\WE;\] and
\item the left-induced model structure exists on $\K$ if and only if \[\rlp{(V^{-1}\Cof)} \subset V^{-1}\WE.\]
\end{enumerate}
\end{prop}

The conditions listed in Proposition \ref{prop:acyclicity-reduction}.(1) and (2) are the \emph{acyclicity conditions} for right- and left-induced model structures.

\begin{proof}
The 2-functor $(-)^\op \colon \Cat^\text{co} \lra \Cat$ exchanges left and right adjoints and cofibrations and fibrations; hence, the statements are dual.\footnote{{The 2-category of categories, functors, and natural transformations has two levels of duals. The superscript $(-)^\text{co}$ signals that the functor $(-)^\op$ acts covariantly on functors but contravariantly on natural transformations.}} By Proposition \ref{prop:model-via-wfs} the right-induced weak factorization systems
\[ \big(\llp\,{U^{-1}\Fib},U^{-1}\Fib\big)\quad \mathrm{and} \quad \big(\llp\,{U^{-1}(\Fib \cap\WE)}, U^{-1}(\Fib \cap\WE)\big)\] define a model structure on $\C$ with weak equivalences $U^{-1}\WE$ if and only if the ``acyclic cofibrations'' are precisely the intersection of the ``cofibrations'' with the weak equivalences, i.e., if and only if
\begin{equation}\label{eq:induced-compatibility} \llp\,{U^{-1}\Fib} = \big(\llp\,{U^{-1}(\Fib \cap\WE)}\big) \cap U^{-1}\WE.\end{equation} As $U^{-1}(\Fib \cap \WE) \subset U^{-1}\Fib$, we have $\llp\,{U^{-1}\Fib} \subset \llp\,{U^{-1}(\Fib \cap \WE)}$. The acyclicity condition, which is clearly necessary, asserts that $\llp\,{U^{-1}\Fib} \subset U^{-1}\WE$. Thus the left-hand side of \eqref{eq:induced-compatibility} is contained in the right-hand side. 

By a standard model categorical argument it follows that the right-hand side is also contained in the left-hand side (see, e.g., the proof of \cite[2.1.19]{hovey}). Given a map $f \in \big(\llp\,{U^{-1}(\Fib \cap\WE)}\big) \cap U^{-1}\WE$, construct a factorization using the weak factorization system $(\llp\,{U^{-1}\Fib},U^{-1}\Fib)$, and arrange the left and right factors as displayed.
\[ \xymatrix{ \bullet \ar[d]_f \ar[r]^{ \in \llp\,{U^{-1}\Fib}} & \bullet \ar[d]^{\in U^{-1}\Fib} \\ \bullet \ar@{=}[r] \ar@{-->}[ur] & \bullet}\] We know that $f$ and the top horizontal map are weak equivalences, so the 2-of-3 property of the class $U^{-1}\WE$ implies that the right vertical map is a member  of $U^{-1}\WE$ as well. The right vertical map therefore lies in $U^{-1}(\Fib \cap \WE)$, which means the displayed diagonal lift exists, whence $f$ is a retract of its left factor and thus a member of the class $\llp{U^{-1}\Fib}$ as claimed.
\end{proof}

\subsection{Techniques for proving acyclicity}\label{subsection:acyclicity}
Quillen proved that a map in a model category $\M$ is a weak equivalence if and only if it is inverted by the localization functor $\M \lra \Ho\M$ \cite[Proposition 5.1]{quillen}. It follows that the class of weak equivalences in a model category always satisfies the \emph{2-of-6 property} (which is stronger than the 2-of-3 property): 
for any composable triple of maps in $\M$
\[ \xymatrix@C=3pc{ A \ar[r]^{f} & B \ar[r]^{g} & C \ar[r]^{h} &D, }\] 
if $gf$ and $hg$ are weak equivalences, then so are $f,g,h$, and $hgf$.

The following theorem combines the 2-of-6 property with the dual of the Quillen Path Object Argument \cite{quillen} to prove acyclicity in various settings.

{\begin{thm}\label{thm:quillen-path} Consider an adjunction between locally presentable categories
\[ \xymatrix@C=4pc{ \K \ar@<1ex>[r]^V \ar@{}[r]|\perp & \M, \ar@<1ex>[l]^K}\]
where $\M$ is an accessible (e.g.,  cofibrantly generated) model category. 
If \begin{enumerate}
\item for every object $X$ in $\K$, there exists a morphism $\epsilon_X \colon QX \rightarrow X$ such that $V\epsilon_{X}$ is a weak equivalence and $V(QX)$ is cofibrant in $\M$, 
\item for each morphism $f\colon X \to Y$ in $\K$  there exists a morphism $Qf\colon QX \to QY$ satisfying $\epsilon _{Y}\circ Qf =f\circ \epsilon_{X}$, and
\item for every object $X$ in $\K$, there exists a factorization $$QX \coprod QX \xrightarrow j\mathrm{Cyl}(QX) \xrightarrow p QX$$  of the fold map such that $Vj$ is a cofibration and $Vp$ is a weak equivalence,
\end{enumerate}
then the acyclicity condition holds for left-induced weak factorization systems on $\K$ and thus the left-induced model structure on $\K$ exists.
\end{thm}}

In particular, (1) holds automatically if all objects are cofibrant in $\M$. 

{\begin{terminology} Abusing language somewhat, we henceforth summarize conditions (1) and (2) above by saying that \emph{$\K$ admits underlying-cofibrant replacements}.  Note that we have deliberately not required functoriality of $Q$, as it is not necessary to the proof below, and moreover does not hold in certain interesting examples.
\end{terminology}}

\begin{proof}
{We will show in Theorem \ref{thm:left-awfs} that the required factorizations, and thus the left-induced weak factorization systems, exist. By Proposition \ref{prop:acyclicity-reduction}, it  remains to prove the acyclicity condition, i.e., } to show  that $\rlp{(V^{-1}\Cof)} \subset V^{-1}\WE$.  

Let ${f: X \lra Y \in  \rlp{(V^{-1}\Cof)}}$. 
First form a lift 
\[ \xymatrix@C=4pc@R=4pc{ \emptyset \ar@{ >->}[d] \ar[r] & QX \ar[r]^\sim_{\epsilon_X} \ar[d]^(.6){Qf} & X \ar[d]^f \\ QY \ar@{=}[r] \ar@{-->}[urr]^(.4)s & QY \ar[r]_\sim^{\epsilon_Y} & Y.}\] 
Next form a lift
\[ \xymatrix@C=4pc@R=4pc{ QX \coprod QX \ar[rr]^{\epsilon_X \coprod (s\circ Qf)} \ar@{ >->}[d] & & X \ar[d]^f \\ \mathrm{Cyl}(QX) \ar[r] \ar@{-->}[urr]^h & QX \ar[r]_{f \circ \epsilon_X} & Y.}\] 
Applying 2-of-3 to the weak equivalence $QX\xrightarrow {\iota_{1}} QX \coprod QX \lra \mathrm{Cyl}(QX)$ and using the fact that $\epsilon_X$ is a weak equivalence, we see that $h$ is a weak equivalence, which implies that $s \circ Qf$ is a weak equivalence. Now apply 2-of-6 to 
\[ \xymatrix@C=4pc@R=4pc{ QX \ar[r]^{s \circ Qf}_{\sim} \ar[d]_{Qf} & X \ar[d]^f \\ QY \ar[r]_{\epsilon_Y}^{\sim} \ar[ur]^s & Y }\] 
to conclude that $f \in V^{-1}\WE$.
\end{proof}

Theorems below illustrate typical applications of Theorem \ref{thm:quillen-path}, which we apply to specific monoidal model categories later in this paper; see Corollary~\ref{cor.of.2.2.3} and Section~\ref{subsec.dg.coring}. 

\begin{thm}\label{thm:a-mod-cylinder} Suppose $(\V,\otimes, \uI)$ is a monoidal model category in which the mon\-oid\-al unit $\uI$ is cofibrant. Let $\M$ be an {accessible} $\V$-model category \cite[4.2.18]{hovey} and a closed monoidal category.  If $A$ is a monoid in $\M$ such that the category $\Mod_A$ of right $A$-modules admits underlying-cofibrant replacements (e.g., if all objects of $\M$ are cofibrant),
then $\Mod_A$ admits a model structure left-induced from the forgetful/hom-adjunction
\[ \xymatrix@C=4pc{ \Mod_A \ar@<1ex>[r]^U \ar@{}[r]|\perp & \M. \ar@<1ex>[l]^{\hom(A,-)}}\]
\end{thm}
\begin{proof}  Since $\M$ is locally presentable, $\Mod_{A}$ is locally presentable as well, as it is a category of algebras over an accessible monad \cite[2.47, 2.78]{adamek-rosicky}. Condition (1) of Theorem \ref{thm:quillen-path} holds by hypothesis. To prove condition (2) of Theorem \ref{thm:quillen-path}, pick any good cylinder object 
\begin{equation}\label{eq:unit-cylinder}  \uI \coprod \uI  \rightarrowtail \mathrm{Cyl}(\uI) \xrightarrow{\sim} \uI\end{equation} 
for the monoidal unit in $\V$. {Note that morphism from $\mathrm{Cyl}(\uI)$ to $\uI$ is in particular a weak equivalence between cofibrant objects, since $\uI$ is cofibrant.} 

The tensor product bifunctor $ \M \times \V \to \M$ lifts to define a tensor product $\Mod_A\times \V \to \Mod_A$ on $A$-modules. For any $A$-module $M$ whose underlying object is cofibrant, applying $M \otimes - \colon \V \to \Mod_A$ to (\ref{eq:unit-cylinder}) defines a good cylinder object on $M$ in $\Mod_A$, {by Ken Brown's lemma and the pushout-product axiom of a $\V$-model category.} The conclusion now follows from Theorem \ref{thm:quillen-path}.
\end{proof}

When the monoid $A$ in  $\M$ is strictly dualizable (i.e., the natural map $A\to \hom \big(\hom (A, \uI), \uI \big)$ is an isomorphism), Theorem \ref{thm:a-mod-cylinder} is a special case of the following general result, which produces left-induced model structures on categories of coalgebras over a cooperad. A review of the basic theory of operads and cooperads can be found in \S\ref{sssec:operad}.  For cooperads, the only applications in this paper are in the setting of chain complexes, but nevertheless we state the following theorem in general.

\begin{thm}\label{thm:cooperad-cylinder} Let $(\V,\otimes, \uI)$ be a monoidal model category  in which the mon\-oid\-al unit $\uI$ is cofibrant, and let $(\M, \sm, \mathbb S)$ be a monoidal category that is also equipped with an {accessible} model structure.  Let $-\boxtimes-: \M \times \V \to \M$ be an op-monoidal functor such that $X\boxtimes \uI \cong X$, and $X\boxtimes-$ preserves finite coproducts, cofibrations, and weak equivalences {between cofibrant objects} whenever $X$ is cofibrant. Let  $\Op Q$ be a cooperad on $\M$ such that the category $\cQ\text{-}\Coalg$ is locally presentable.  {Let $\Op Q'$ be a cooperad in $\V$ equipped with a map $\Op Q \boxtimes^{*} \Op Q' \lra \Op Q$ of cooperads in $\M$, where $\boxtimes^{*}$ denotes here the functor given by applying $\boxtimes$ levelwise,} and such that $\uI$ admits the structure of a $\Op Q'$-coalgebra, extending to 
\begin{equation}\label{eq:unit-cylinder2}  \uI \coprod \uI   \rightarrowtail \mathrm{Cyl}(\uI) \xrightarrow{\sim} \uI\end{equation}
in {$\cQ'\text{-}\Coalg_{\V}$}, where cofibrations and weak equivalences are created in $\V$. 

If $\cQ\text{-}\Coalg$ admits underlying-cofibrant replacements (e.g., if all objects of $\M$ are cofibrant), 
 then it admits a model structure that is left-induced via the forgetful/cofree adjunction
\[ \xymatrix@R=4pc@C=4pc{ \cQ\text{-}\Coalg \ar@<1ex>[r]^-V \ar@{}[r]|-\perp & \M. \ar@<1ex>[l]^-{\Gamma_\cQ}}\] 
\end{thm}

\begin{rmk}  It follows from \cite[Proposition A.1]{ching-riehl} that if $V\Gamma_{\cQ}$ is an accessible endofunctor, then $\cQ\text{-}\Coalg$ is locally presentable, since $\M$ is locally presentable.
\end{rmk}

\begin{rmk} If $\M$ is an {accessible} $\V$-model category, then the hypotheses in the first two sentences of Theorem \ref{thm:cooperad-cylinder} are all satisfied. The reason for enumerating them individually is that certain interesting examples, such as those appearing in Theorem \ref{thm:cooperadsQQ'}, satisfy these conditions without being $\V$-model categories.
\end{rmk}

\begin{proof}
The first condition of Theorem \ref{thm:quillen-path} holds by the hypothesis on underlying-cofibrant replacement. The remaining hypotheses combine to prove the second condition, as follows. {The morphism $\Op Q \boxtimes^{*} \Op Q' \lra \Op Q$ induces} a bifunctor 
$$ \cQ\text{-}\Coalg \times \cQ'\text{-}\Coalg_{\V} \to  \cQ\text{-}\Coalg.$$
If $C$ is a $\cQ$-coalgebra whose underlying object is cofibrant, then applying $C\boxtimes-$ to \eqref{eq:unit-cylinder2}  defines a good cylinder object for $C \in \cQ\text{-}\Coalg$.  Explicitly,  if $\big(C, \{\rho_{n}\}_{n}\big)$ is a $\cQ$-coalgebra, then applying $C\boxtimes -$ to \eqref{eq:unit-cylinder2} gives
$$C\coprod C \rightarrowtail C\boxtimes\mathrm{Cyl}(\uI) \xrightarrow{\sim} C,$$
which lifts to $\cQ\text{-}\Coalg$ as follows.  The components of the $\cQ$-comultiplication on $C\otimes \mathrm{Cyl}(\uI) $ are given by the composites
{\small\[\xymatrix@C=15pt{C \boxtimes \mathrm{Cyl}(\uI)  \ar [rr]^-{\rho_{n}\boxtimes \rho'_{n}} &&\big(\cQ(n) \sm C^{\sm n}\big)\boxtimes\big (\cQ'(n) \otimes \mathrm{Cyl}(\uI) ^{\otimes n}\big)\ar[d]^{\tau}& \\
&& \big(\cQ(n) \boxtimes \cQ'(n)\big)\sm (C\boxtimes \mathrm{Cyl}(\uI) )^{\sm n}\ar [rr]^-{\varphi(n)\sm \id} &&\cQ(n) \sm (C\boxtimes \mathrm{Cyl}(\uI) )^{\otimes n},}\]}

\noindent where $\tau$ is the natural transformation encoding the op-monoidality of the functor $-\boxtimes -$. 

Existence of the desired model structure follows from Theorem \ref{thm:quillen-path}.
\end{proof}

\subsection{Adjoint squares}\label{section:adjoint_squares}

Let  $\TT=(T,\eta,\mu)$ be a monad and $\KK = (K, \epsilon,\delta)$ a comonad on a locally presentable category $\M$, equipped with a  distributive law $\chi\colon TK \Rightarrow KT$.  Recall that a \emph{distributive law} of a monad over a comonad is a natural transformation $\chi$ so that the diagrams 
\begin{equation}\label{eq:distributive-law} \vcenter{ \xymatrix@=10pt{ &&&  & & TKT \ar[drr]^{\chi T} \\ &  K \ar[ddl]_{\eta K} \ar[ddr]^{K\eta} && T^2 K \ar[urr]^{T\chi} \ar[ddr]_{\mu K} & & & & KT^2 \ar[ddl]^{K\mu} \\ \\  TK \ar[rr]^{\chi} \ar[ddr]_{T\epsilon} & & KT  \ar[ddl]^{\epsilon T}&& TK \ar[rr]^{\chi} \ar[ddl]_{T\delta} & & KT \ar[ddr]^{\delta T} \\ \\ & T& &  TK^2 \ar[drr]_{\chi K} & & & & K^2T   \\ &&&  & & KTK \ar[urr]_{K\chi}}}  
\end{equation} 
commute.
Let $\Alg^\TT(\M)$, $\Coalg_\KK(\M)$, and $\Bialg_\KK^\TT(\M)$ denote the categories of $\TT$-algebras, $\KK$-coalgebras, and $(\TT, \KK)$-bialgebras, respectively.  Note that the existence of the distributive law $\chi$ ensures that we can make sense of the notion of $(\TT, \KK)$-bialgebras.

There is an associated diagram of adjunctions
\[ \xymatrix@C=4pc@R=4pc{ \M \ar@{}[r]|{\perp} \ar@<-1ex>[d]_T \ar@<-1ex>[r]_{K} & \Coalg_\KK(\M)  \ar@<1ex>[d]^T \ar@<-1ex>[l]_V \\  \ar@{}[r]|{\top} \ar@{}[u]|{\dashv} \Alg^\TT(\M) \ar@<1ex>[r]^{K} \ar@<-1ex>[u]_U & \Bialg_\KK^\TT(\M) \ar@<1ex>[u]^U \ar@<1ex>[l]^V\ar@{}[u]|{\vdash} }\] 
in which the square of forgetful functors commutes, and in which the squares of left adjoints and of right adjoints commute up to natural isomorphism: $UV=VU$, $TV \cong VT$, and $UK \cong KU$. Here  $U$ and $V$ denote forgetful functors, $T$ is the free algebra functor, and  $K$ is the cofree coalgebra functor.

In general settings of this nature, the following result provides conditions under which the model structure can be transferred from $\M$ to the category $ \Bialg_\KK^\TT(\M)$ using the adjunctions above.

\begin{thm}\label{thm:square} Given a square of adjunctions 
\[ \xymatrix@R=4pc@C=4pc{ \K \ar@{}[r]|{\perp} \ar@<-1ex>[d]_L \ar@<-1ex>[r]_{R} & \M  \ar@<1ex>[d]^L \ar@<-1ex>[l]_V \\  
\ar@{}[r]|{\top} \ar@{}[u]|{\dashv} \N \ar@<1ex>[r]^{R} \ar@<-1ex>[u]_U & \P \ar@<1ex>[u]^U \ar@<1ex>[l]^V\ar@{}[u]|{\vdash} }\] 
between locally presentable categories, suppose that $\K$ has a model structure $(\Cof,\Fib,\WE)$ such that the right-induced model structure exists on $\N$, created by $U\colon \N \lra \K$, and that the left-induced model structure exists on $\M$, created by $V\colon \M \lra \K$. 

If $UV\cong VU$ and if $LV\cong VL$ (or, equivalently, $UR\cong RU$), then there exists a right-induced model structure on $\P$, created by $U\colon \P \lra \M$,  and a left-induced model structure on $\P$, created by $V\colon \P \lra \N$, so that the identity {is the left member of a Quillen equivalence from the right-induced model structure to the left-induced one:}
\[ \xymatrix@C=4pc@R=4pc{  \P_{\mathrm{right}} \ar@<1ex>[r]^-\id \ar@{}[r]|-\perp & \P_{\mathrm{left}}. \ar@<1ex>[l]^-\id}\]
\end{thm}

\begin{proof}
Denote the weak factorization systems in the left- and right-induced model structures on $\P$ by
\begin{align*} &\P_{\mathrm{left}}: && \big(V^{-1}\Cof_\N, \rlp{(V^{-1}\Cof_\N)} \big) &\text{} & \big(V^{-1}(\Cof_\N \cap \WE_\N), \rlp{(V^{-1}(\Cof_\N \cap \WE_\N))} \big) \\  &\P_{\mathrm{right}}: &&  \big(\llp\,{U^{-1}(\Fib_\M\cap\WE_\M)}, U^{-1}(\Fib_\M\cap\WE_\M) \big) &\text{} & \big(\llp\,{U^{-1}\Fib_\M}, U^{-1}\Fib_\M \big) 
\end{align*} 
The left classes in the weak factorization systems on $\P_{\mathrm{left}}$ are created by the left adjoint $V \colon \P_{\mathrm{left}} \lra \N$. Dually, the right classes in the weak factorization systems on $\P_{\mathrm{right}}$ are created by the right adjoint $U \colon \P_{\mathrm{right}} \lra \M$.

Thus, any functor whose codomain is $\P_{\mathrm{left}}$ preserves left classes of weak factorization systems if and only if the composite with $V$ does so. By a standard adjoint lifting property argument, the composite $P_{\mathrm{right}} \xrightarrow{\id} P_{\mathrm{left}} \xrightarrow{V} \N$ preserves the left classes of weak factorization systems if and only if its right adjoint $\N \xrightarrow{R} \P_{\mathrm{left}} \xrightarrow{\id} P_{\mathrm{right}}$ preserves right classes. But because $U$ creates the right classes, this is the case if and only if the composite right adjoint
\[ \xymatrix@C=4pc@R=4pc{ \M \ar@<1ex>[r]^-L \ar@{}[r]|-{\perp} & \P_{\text{right}} \ar@<1ex>[l]^-{U} \ar@<1ex>[r]^-\id \ar@{}[r]|-\perp & \P_{\text{left}} \ar@<1ex>[l]^-\id \ar@<1ex>[r]^-V \ar@{}[r]|-{\perp} & \N \ar@<1ex>[l]^-{R}}\]  preserves the right classes. Using $UR\cong RU$, this follows from the fact that $RU$ is right Quillen (or using $VL\cong LV$, this follows from the fact that $LV$ is left Quillen and the adjoint lifting property argument). In conclusion, the left classes of the right-induced weak factorization systems on $\P$ are contained in the left classes of the left-induced weak factorization systems, and dually, the right classes of the left-induced weak factorization systems are contained in the right classes of the right-induced weak factorization systems.

To establish the model structures it remains to check the acyclicity conditions. For $\P_{\mathrm{right}}$, we want to show that 
$$\llp\,{U^{-1}\Fib_\M}\subset U^{-1}\WE_\M = U^{-1}V^{-1}\WE = V^{-1}U^{-1}\WE,$$
which is a consequence of 
$$\llp\,{U^{-1}\Fib_\M} \subset V^{-1}(\Cof_\N \cap \WE_\N) \subset V^{-1}\WE_\N = V^{-1}U^{-1}\WE,$$
which we established above.

Dually, for $\P_{\mathrm{left}}$, we want to show that 
$$ \rlp{(V^{-1}\Cof_\N)}\subset V^{-1}\WE_\N =V^{-1}U^{-1}\WE= U^{-1}V^{-1}\WE,$$ 
which follows from
$$ \rlp{(V^{-1}\Cof_\N)} \subset U^{-1}(\Fib_\M\cap\WE_\M) \subset U^{-1}\WE_\M = U^{-1}V^{-1}\WE,$$
again established above.
\end{proof}

\begin{rmk} We have not yet found appealing and useful sufficient conditions guaranteeing that the right and left model structures defined above are the same, but we do provide examples in this paper where the structures are distinct (Proposition \ref{prop:distinct}) and where they are the same (Proposition \ref{prop:reedy_m_s}).
\end{rmk}

\begin{rmk}
By a theorem of Jardine \cite{Jardine}, if there are two model structures with the same weak equivalences and $\Cof_1 \subset \Cof_2$ (hence $\Fib_2 \subset \Fib_1$), then for any intermediate weak factorization system $(\cL,\cR)$, meaning $\Cof_1 \subset \cL \subset \Cof_2$, there is a third model structure with the same weak equivalences and with cofibrations equal to $\cL$ (or dually for a weak factorization system intermediate to the fibrations).  It follows that if the two model structures on $\P$ created in Theorem \ref{thm:square} are indeed distinct, then they may give rise to an interesting family of intermediate model structures. 
\end{rmk}

\begin{cor}\label{cor:monad_comonad} Let $\M $ be a locally presentable model category, and let $\TT$ and $\KK$ be an accessible monad  and an accessible comonad on $\M$ such that there exists a  distributive law $TK \Rightarrow KT$. If $V\colon \Coalg_\KK(\M) \lra \M$ creates a left-induced model structure on $\Coalg_\KK(\M)$,   and $U\colon \Alg^\TT(\M) \lra \M$ creates a right-induced model structure on $\Alg^\TT(\M)$, then there exist left- and right-induced model structures on the category $\Bialg_\KK^\TT(\M)$, created by $V\colon \Bialg_\KK^\TT(\M) \lra  \Alg^\TT(\M)$ and $U\colon  \Bialg_\KK^\TT(\M) \lra  \Coalg_\KK(\M)$, respectively.
\end{cor}

\section{Left and right induced algebraic weak factorization systems}\label{sec:induced_awfs}

A model structure on a category is comprised of a pair of interacting weak factorization systems: one whose left class defines the acyclic cofibrations and whose right class defines the fibrations and another whose left class is the cofibrations and whose right class is the acyclic fibrations. With the exception of a few, extremely pathological situations, these weak factorization systems can be promoted to \emph{algebraic weak factorization systems}, highly structured objects with superior categorical properties. An algebraic weak factorization system is \emph{accessible} if the underlying category is locally presentable and if its functorial factorization is accessible (preserves sufficiently large filtered colimits).

As shown by Bourke and Garner \cite{bourke-garner}, any accessible algebraic weak factorization system can be ``left-induced'' along any left adjoint between locally presentable categories and ``right-induced'' along any right adjoint between locally presentable categories. The resulting algebraic weak factorization systems in each case are again accessible, so this induction process can be repeated as necessary.

Cofibrantly generated weak factorization systems are a special case of accessible algebraic weak factorization systems. In this case, the right induction is classical, and the left induction is a recent result of Makkai-Rosick\'{y} \cite{makkai-rosicky}. Our particular interest here is in accessible algebraic weak factorization systems that are not cofibrantly generated, at least in the classical sense. Our primary source of examples will be \emph{enriched cofibrantly generated} algebraic weak factorization systems introduced in \cite[Chapter 13]{riehl-cathtpy}. Classical examples of these include the weak factorization systems for the Hurewicz model structures on chain complexes over a commutative ring or modules for a differential graded algebra in such. Other examples include certain weak factorization systems on diagram categories. These will be discussed in Section \ref{sec:enriched-cof-gen}.

In this section, we give an expository presentation of the theory of accessible algebraic weak factorization systems and recent work of Bourke and Garner, which we expect to be totally unfamiliar to most readers. In \S\ref{ssec:accessible}, we introduce algebraic weak factorization systems and accessible model categories. In \S\ref{ssec:recognizing}, we describe a key theorem of Bourke and Garner, which allows one to recognize algebraic weak factorization systems ``in the wild'' --- in particular, without presenting an explicit functorial factorization. A straightforward application of this theorem, reviewed in \S\ref{ssec:inducing}, allows Bourke and Garner to prove that accessible algebraic weak factorization systems can be left- or right-induced along any adjunction between locally presentable categories. The Acyclicity Theorem, mentioned in the introduction, is a trivial corollary of their result: in the case of an adjunction between a locally presentable category and an accessible model category, this result implies that left- or right-induced weak factorization systems always exist. {The one original contribution in this section appears in \S\ref{ssec:injective}, where we apply the Acyclicity Theorem to prove that the category of $\D$-indexed diagrams valued in an accessible model category admits both projective and injective model structures, modernizing an argument that first appeared in \cite{riehl-algebraic}.}

\subsection{Algebraic weak factorization systems and accessible model categories}\label{ssec:accessible}

A weak factorization system $(\cL,\cR)$ on a category $\D$ underlies an \emph{algebraic weak factorization system} if it admits a particularly nice functorial factorization. A \emph{functorial factorization} is given by a pair of endofunctors $L,R \colon \D^\2 \lra \D^\2$ of the category of arrows and commutative squares in $\D$ so that $\mathrm{dom}\cdot L = \mathrm{dom}$, $\mathrm{cod}\cdot R = \mathrm{cod}$, $\mathrm{cod}\cdot L = \mathrm{dom} \cdot R$, and $f = Rf \cdot Lf$ for any $f \in \mor\D$. These endofunctors are ``pointed'': there exist natural transformations $\vec{\epsilon} \colon L \Rightarrow \id$ and $\vec{\eta} \colon \id \Rightarrow R$ whose components rearrange the factorization $f = Rf \cdot Lf$ around the edges of a commutative square:
\begin{equation}\label{eq:pointed-factorization} \xymatrix{ \ar@{}[d]|{\displaystyle\vec{\epsilon}_f :=} & \cdot \ar[d]_{Lf} \ar@{=}[r] & \cdot \ar[d]^f \\ &  \cdot \ar[r]_{Rf} & \cdot}\qquad \qquad \xymatrix{\ar@{}[d]|{\displaystyle\vec{\eta}_f :=} &  \cdot \ar[r]^{Lf} \ar[d]_f & \cdot \ar[d]^{Rf} \\ &  \cdot \ar@{=}[r] & \cdot}\end{equation}

\begin{defn} An \emph{algebraic weak factorization system} $(\LL,\RR)$ on $\D$ consists of a comonad $\LL  = (L,\vec{\epsilon},\vec{\delta})$ and a monad $\RR = (R,\vec{\eta}, \vec{\mu})$ on $\D^\2$ whose underlying pointed endofunctors define a functorial factorization and in which the canonical map $LR \Rightarrow RL$ is a distributive law.\footnote{Here the comonad distributes over the monad, and a distributive law is a natural transformation satisfying diagrams dual to those of \eqref{eq:distributive-law}.}
\end{defn}

The comonad and monad of an algebraic weak factorization system $(\LL,\RR)$ have associated categories of coalgebras $U\colon \coalg{\LL} \lra \D^\2$ and algebras $U \colon \alg{\RR} \lra \D^\2$, equipped with forgetful functors to the category of arrows (drawn ``vertically'') and commutative squares in $\D$. The $\LL$-coalgebras and $\RR$-algebras give an algebraic presentation of the left and right classes of the underlying ordinary weak factorization system, which are defined to be the retract closures of the images of these forgetful functors.  Importantly, an $\LL$-coalgebra lifts against any $\RR$-algebra in a canonical way, and these canonical solutions to lifting problems are preserved by $\LL$-coalgebra and $\RR$-algebra morphisms. In this sense, the lifting properties of the weak factorization system have been made ``algebraic.''

\begin{defn} An algebraic weak factorization system $(\LL,\RR)$ on $\D$ is \emph{accessible} if $\D$ is locally presentable and if the functors $L,R \colon \D^\2 \lra \D^\2$ are accessible, i.e., preserve $\lambda$-filtered colimits for some regular cardinal $\lambda$.
\end{defn}

 Because colimits in the diagram category $\D^\2$ are defined objectwise, a functorial factorization $L,R$, such as provided by an algebraic weak factorization system, is accessible if and only if the functor $E = \mathrm{cod}\cdot L = \mathrm{dom}\cdot R \colon \D^\2 \lra \D$ that sends an arrow to the object through which it factors is accessible.

\begin{prop}[{Garner \cite[16]{bourke-garner}}]\label{prop:garner-SOA} Any set of arrows $\J$ in a locally presentable category $\D$ generates an accessible algebraic weak factorization system whose underlying ordinary weak factorization system is the usual weak factorization system cofibrantly generated by $\J$.
\end{prop}
\begin{proof}[Proof sketch]
In  a locally presentable category, all objects are small: every covariant representable functor preserves filtered colimits for a sufficiently large ordinal. Taking advantage of this smallness condition, Garner shows that a modified version of Quillen's small object argument, which avoids ``attaching redudant cells,'' eventually converges to define the functorial factorization for an algebraic weak factorization system $(\LL,\RR)$  cofibrantly generated by $\J$ \cite[4.22]{garner-understanding}. It follows that the functorial factorization so-constructed is accessible; see \cite[16]{bourke-garner} or  \cite[1.3]{ching-riehl}.

Here \emph{cofibrant generation} means that the category $\alg{\RR}$ is isomorphic over $\D^\2$ to the category $\J^\boxslash$ whose objects are morphisms in $\D$ together with specified solutions to any lifting problem against any morphism in $\J$ and whose maps are commutative squares preserving these chosen solutions. This implies in particular that the underlying weak factorization system $(\cL,\cR)$ of $(\LL,\RR)$ is the one cofibrantly generated by $\J$. 
\end{proof}

\begin{rmk} Proposition \ref{prop:garner-SOA} applies more generally to any small category $\J$ of generating arrows, in which case the objects of $\J^\boxslash\cong\alg{\RR}$ are morphisms in $\D$ equipped with solutions to all lifting problems against objects of $\J$ that are natural with respect to the morphisms in $\J$.
\end{rmk}

\begin{defn} An \emph{accessible model structure} is a model structure $(\M,\Fib,\Cof,\WE)$ so that $\M$ is locally presentable and the weak factorization systems $(\Cof \cap \WE,\Fib)$ and  $(\Cof ,\Fib \cap \WE)$ are accessible.\end{defn}

As an immediate corollary of Proposition \ref{prop:garner-SOA} we obtain the following important class of examples.

\begin{cor} Any combinatorial model structure can be equipped with a pair of accessible algebraic weak factorization systems, and thus defines an accessible model structure.
\end{cor}

\begin{rmk} A priori our definition of accessible model category is more rigid than the one introduced by Rosick\'y in \cite{rosicky-accessible}, who asks only  that the constitutent weak factorization systems admit functorial factorizations that are accessible. However, he shows that any accessible weak factorization system on a locally presentable category is cofibrantly generated by a small category. In this context, Garner's small object argument produces an accessible weak factorization system, so the two definitions in fact agree.
\end{rmk}

\subsection{Recognizing algebraic weak factorization systems}\label{ssec:recognizing}

The $\LL$-coalgebras and $\RR$-algebras most naturally assemble into (strict) double categories $\CCoalg{\LL}$ and $\AAlg{\RR}$. Objects and horizontal arrows are just objects and arrows of $\D$. Vertical arrows are $\LL$-coalgebras (respectively, $\RR$-algebras) and squares are commutative squares (maps in $\D^\2$) that lift to maps of $\LL$-coalgebras (respectively, $\RR$-algebras). These double categories are equipped with forgetful double functors $\UU \colon \CCoalg{\LL} \lra \Sq{\D}$ and $\UU \colon \AAlg{\LL} \lra \Sq{\D}$ to the double category of objects, morphisms, morphisms, and commutative squares in $\D$. Recalling that double categories are categories internal to $\CAT$, the diagrams displayed on the left and on the right below encode the data of the forgetful double functors $\CCoalg{\LL} \lra \Sq{\D}$ and $\AAlg{\LL} \lra \Sq{\D}$ respectively
\[ \xymatrix@C=4pc@R=2pc{ \CCoalg{\LL} \ar[r]^\UU & \Sq{\D} & \AAlg{\RR} \ar[r]^\UU & \Sq{\D} \\
\coalg{\LL} \pbtimes{\D} \coalg{\LL} \ar[d]_{\circ} \ar[r]^-{U \pbtimes{1} U} & \D^\3 \ar[d]^\circ & \alg{\RR} \pbtimes{\D} \alg{\RR} \ar[r]^-{U \pbtimes{1} U} \ar[d]_\circ & \D^\3 \ar[d]^\circ \\ \coalg{\LL} \ar@<-1ex>[d]_s \ar@<1ex>[d]^t \ar@{<-}[d]|i \ar[r]^U & \D^\2 \ar@<-1ex>[d]_s \ar@<1ex>[d]^t \ar@{<-}[d]|i  & \alg{\RR} \ar@<-1ex>[d]_s \ar@<1ex>[d]^t \ar@{<-}[d]|i  \ar[r]^U & \D^\2 \ar@<-1ex>[d]_s \ar@<1ex>[d]^t \ar@{<-}[d]|i  \\ \D \ar[r]_1 & \D &\D \ar[r]_1 & \D }\]

\begin{ex} If $(\LL,\RR)$ is a cofibrantly generated algebraic weak factorization system (produced by the Garner small object argument applied to a set or small category $\J$), then $\AAlg{\RR}$ is isomorphic to the double category $\J^\boxslash$ whose vertical arrows are maps with chosen (natural) solutions to lifting problems against $\J$ and whose commutative squares preserve such chosen lifts.
\end{ex}

Conversely, each of the double categories $\coalg{\LL}$ and $\alg{\RR}$ encodes the entire structure of the algebraic weak factorization system. There are many expressions of this fact, but the most useful for our purposes is the following theorem of Bourke and Garner. A double category with a forgetful double functor to $\Sq\D$ is \emph{concrete} if the object component of this functor is an identity and if the arrow component is faithful. For example, the double categories $\CCoalg{\LL}$ and $\AAlg{\RR}$ are concrete:  the arrow components are the faithful functors $U \colon \coalg{\LL} \lra \D^\2$ and $U \colon \alg{\RR} \lra \D^\2$.

\begin{thm}[{\cite[6]{bourke-garner}}]\label{thm:bourke} A concrete double category $\AA$ with a forgetful double functor $\UU \colon \AA \lra \Sq\D$ encodes the double category of coalgebras for an algebraic weak factorization system on $\D$  if and only if 
\begin{enumerate}
\item[(i)] the functor $U \colon \A_1 \lra \D^\2$ on arrows is comonadic, and
\item[(ii)] for every vertical arrow $f \colon a \lra b$ in $\AA$, the square 
\[ \xymatrix{ a \ar[r]^1 \ar[d]_1 & a \ar[d]^f \\ a \ar[r]_f & b}\] is in $\AA$.
\end{enumerate} Dually, a concrete double category $\AA$ with a forgetful double functor $\UU \colon \AA \lra \Sq\D$  encodes the double category of algebras for an algebraic weak factorization system on $\D$  if and only if 
\begin{enumerate}
\item[(i)] the functor $U \colon \A_1 \lra \D^\2$ on arrows is monadic, and
\item[(ii)] for every vertical arrow $f \colon a \lra b$ in $\AA$, the square 
\[ \xymatrix{ a \ar[r]^f \ar[d]_f & b \ar[d]^1 \\ b \ar[r]_1 & b}\] is in $\AA$.
\end{enumerate}
\end{thm}

The advantage of Theorem \ref{thm:bourke} is that it makes it possible to recognize algebraic weak factorization systems ``in the wild'' --- in particular, without a specific functorial factorization in mind. Here is how the functorial factorization is recovered. Using condition (i), the monadic functor $U \colon \A_1 \lra \D^\2$ induces a monad $R' \colon \D^\2 \lra \D^2$, the composite of $U$ with its left adjoint. By an elementary categorical lemma that makes use of condition (ii), $R'$ is  isomorphic to a monad $R \colon \D^\2 \lra \D^2$ with the property that $\mathrm{cod} \cdot R = \mathrm{cod}$. In particular, $R$ is ``pointed,'' encoding a functorial factorization whose left factor can be recovered from the unit component, as displayed in \eqref{eq:pointed-factorization}.

\subsection{Left and right induction}\label{ssec:inducing}

The following corollaries of Theorem \ref{thm:bourke}, observed by Bourke and Garner, demonstrate the existence of left- and right-induced algebraic weak factorization systems along adjunctions between locally presentable categories. These will be used to construct new model structures of interest.

\begin{thm}[Left induction {\cite[13]{bourke-garner}}]\label{thm:left-awfs} Let $K \colon \C \lra \D$ be a cocontinuous functor between locally presentable categories, and suppose that $\D$ has an accessible algebraic weak factorization system $(\LL,\RR)$. Then the pullback
\[ \xymatrix{ \AA \ar[d]_{\tilde{\UU}} \ar[r]^-{\tilde{K}} \ar@{}[dr]|(.2){\lrcorner} & \CCoalg{\LL} \ar[d]^\UU \\ \Sq{\C} \ar[r]_{\Sq{K}} & \Sq{\D}}\] encodes an accessible algebraic weak factorization system on $\C$, whose underlying weak factorization system is left-induced from the underlying weak factorization system on $\D$.
\end{thm}

The special adjoint functor theorem implies that a cocontinuous functor between locally presentable categories is a left adjoint.

\begin{thm}[Right induction {\cite[13]{bourke-garner}}]\label{thm:right-awfs}  Let $K \colon \C \lra \D$ be an accessible and continuous functor between locally presentable categories, and suppose that $\D$ has an accessible algebraic weak factorization system $(\LL,\RR)$. Then the pullback
\[ \xymatrix{ \AA \ar[d]_{\tilde{\UU}} \ar[r]^-{\tilde{K}} \ar@{}[dr]|(.2){\lrcorner} & \AAlg{\RR} \ar[d]^\UU \\ \Sq{\C} \ar[r]_{\Sq{K}} & \Sq{\D}}\] encodes an accessible algebraic weak factorization system on $\C$,  whose underlying weak factorization system is right-induced from the underlying weak factorization system on $\D$.
\end{thm}

Note that a functor between locally presentable categories is accessible and continuous if and only if it is a right adjoint \cite[1.66]{adamek-rosicky}.

\begin{rmk} If $(\LL,\RR)$ is cofibrantly generated, so that $\AAlg{\RR} \cong \J^\boxslash$ for some small category $\J$, and if $K \colon \C \lra \D$ is right adjoint to $F \colon \D \lra \C$, then $\AA \cong (F\J)^\boxslash$, so that the right-induced algebraic weak factorization system is again cofibrantly generated.

In fact, all accessible algebraic weak factorization systems are ``cofibrantly generated'' --- provided that this is meant in an expanded sense that includes the possibility of generating not only by a small category $\J$ of arrows but by a small \emph{double category} of arrows. Classical cofibrant generation is the special case in which the double category is a discrete category of vertical arrows, the ``generating cofibrations.'' We enthusiastically refer the reader to \cite[25]{bourke-garner}; however, we won't make use of this fact here.
\end{rmk}

Theorems \ref{thm:left-awfs} and \ref{thm:right-awfs} lead to an improved version of Proposition \ref{prop:acyclicity-reduction}, i.e. the Acyclicity Theorem mentioned in the introduction.

\begin{cor}\label{cor:acyclicity-reduction}
Suppose $(\M, \Fib, \Cof,\WE)$ is an accessible model category, $\C$ and $\K$ are locally presentable, and there exist adjunctions 
\[ \xymatrix@C=4pc@R=4pc{ \K \ar@<1ex>[r]^V \ar@{}[r]|\perp & \M \ar@<1ex>[l]^R \ar@<1ex>[r]^L\ar@{}[r]|\perp & \C. \ar@<1ex>[l]^U}\]  
\begin{enumerate}
\item The right-induced model structure exists on $\C$ if and only if \[\llp\,{U^{-1}\Fib} \subset U^{-1}\WE.\]
\item The left-induced model structure exists on $\K$ if and only if \[\rlp{(V^{-1}\Cof)} \subset V^{-1}\WE.\]
\end{enumerate}
{When these conditions are satisfied, the induced model structures on $\C$ and on $\K$ are again accessible.}
\end{cor}

{
\subsection{Projective and injective accessible model structures}
\label{ssec:injective}

For any small category $\D$ and bicomplete category $\M$, the forgetful functor $\M^\D \lra \M^{\ob\D}$ is both monadic and comonadic, with left and right adjoints given by left and right Kan extension. If $\M$ is a model category, then $\M^{\ob\D} \cong \prod_{\ob\D}\M$ inherits a pointwise-defined model structure, with weak equivalences, cofibrations, and fibrations all defined pointwise in $\M$. Assuming they exist, the right-induced model structure on $\M^\D$ is called the \emph{projective model structure}, while the left-induced model structure is called the \emph{injective model structure}.

In this section we apply Corollary \ref{cor:acyclicity-reduction} to prove the following existence result for model structures on diagram categories.

\begin{thm} If $\M$ is an accessible model category, and $\D$ is any small category, then $\M^\D$ admits both injective and projective model structures, which are again accessible.
\end{thm}

In the case where the model structure on $\M$ is given by a pair of cofibrantly generated algebraic weak factorization systems, the argument given here first appeared as \cite[Theorem 4.5]{riehl-algebraic}.

\begin{proof}
We prove that $\M^\D$ admits the injective model structure, left-induced from the pointwise defined model structure on $\M^{\ob\D}$. The proof of the existence of the projective model structure is dual. 

To apply Corollary \ref{cor:acyclicity-reduction} to the restriction--right Kan extension adjunction
\[ \xymatrix@C=4pc@R=4pc{ \M^\D \ar@<1ex>[r]^V \ar@{}[r]|\perp & \M^{\ob\D}, \ar@<1ex>[l]^R }\]  
it remains only to check the acyclicity condition (2), which we prove using an algebraic argument.

Let $(\CC,\FF_t)$ denote the algebraic weak factorization system for the cofibrations and trivial fibrations in $\M$. The categories $\M^\D$ and $\M^{\ob\D}$ both inherit pointwise-defined algebraic weak factorization systems $(\CC^\D,\FF_t^\D)$ and $(\CC^{\ob\D}, \FF_t^{\ob\D})$, given by postcomposing with the comonad $\CC$ and monad $\FF_t$. Write $(\CC^\text{inj},\FF_t^\text{inj})$ for the algebraic weak factorization system on $\M^\D$ that is left induced from $(\CC^{\ob\D}, \FF_t^{\ob\D})$, applying Theorem \ref{thm:left-awfs}. If we write $\Cof$ and $\WE$ for the pointwise-defined cofibrations and weak equivalences in $\M^{\ob\D}$, then the underlying weak factorization system of $(\CC^\text{inj},\FF_t^\text{inj})$ is $(V^{-1}\Cof,\rlp{(V^{-1}\Cof)})$, the weak factorization system for the cofibrations and weak equivalences for the injective model structure.

Because the class of weak equivalences is retract-closed, to prove that $\rlp{(V^{-1}\Cof)} \subset V^{-1}\WE$, it suffices to show that every $\FF_t^\text{inj}$-algebra is an $\FF_t^\D$-algebra, as the components of an $\FF_t^\D$-algebra are $\FF_t$-algebras, in particular acyclic fibrations and thus weak equivalences. To establish this inclusion, it is easiest to argue from the other side. Because restriction along $\ob\D\hookrightarrow\D$ commutes with postcomposition with the comonad $\CC$, there exist canonically defined double functors
\[ \xymatrix{\CCoalg{\CC^{\D}} \ar@/^/[drr] \ar@/_/[ddr] \ar@{-->}[dr] & \\ &  \CCoalg{\CC^{\text{inj}}} \ar[d]_{\tilde{\UU}} \ar[r]^-{\tilde{K}} \ar@{}[dr]|(.2){\lrcorner} & \CCoalg{\CC^{\ob\D}} \ar[d]^\UU \\ & \Sq{\M^\D} \ar[r]_{\Sq{K}} & \Sq{\M^{\ob\D}}}\] inducing a double functor $\CCoalg{\CC^{\D}} \to  \CCoalg{\CC^{\text{inj}}}$ by the universal property that defines the left-induced algebraic weak factorization system.\footnote{A $\CC^\text{inj}$-coalgebra is a natural transformation whose components are $\CC$-coalgebras, while a $\CC^\D$-coalgebra is a natural transformation whose components are $\CC$-coalgebras and whose naturality squares are $\CC$-coalgebra morphisms. The induced map is the natural inclusion of the class of $\CC^\D$-coalgebras into the class of $\CC^\text{inj}$-coalgebras.} Such a double functor encodes a \emph{morphism of algebraic weak factorization systems} $(\CC^\D,\FF_t^\D) \to (\CC^{\text{inj}},\FF_t^{\text{inj}})$, the right-hand component of which defines the inclusion $\FF_t^{\text{inj}}$-algebras into $\FF_t^{\D}$-algebras.\footnote{A morphism of algebraic weak factorization systems is a special case of an \emph{adjunction of algebraic weak factorization systems}, the theory of which is developed in \cite{riehl-algebraic} and \cite{bourke-garner}.} This proves  the acyclicity condition.
\end{proof}

}

\section{Enriched cofibrantly generated algebraic weak factorization systems}\label{sec:enriched-cof-gen}

We introduced our notion of accessible model category because  of our particular interest in model structures that are known not to be cofibrantly generated, at least in the classically understood sense, but that are accessible. We shall see, in Theorem \ref{thm:acc-awfs} below, that any enriched cofibrantly generated weak factorization system on a locally presentable category can be promoted to an accessible algebraic weak factorization system, to which Theorems \ref{thm:left-awfs} and \ref{thm:right-awfs} apply.

In \S\ref{ssec:enriched-gen}, we briefly sketch the ideas behind the notion of enriched cofibration. These are included for context and because we expect they will be of interest to many readers, but are not essential to the present work. In \S\ref{ssec:enriched-SOA}, we prove that the enriched algebraic small object argument produces accessible algebraic weak factorization systems; the precise statement of this result may be found in Theorem \ref{thm:acc-awfs}. This result provides the major source of accessible model categories that are not combinatorial.

\subsection{Enriched cofibrant generation}\label{ssec:enriched-gen}

Before explaining the notion of enriched cofibrant generation, we whet the reader's appetite with a pair of classical examples.

\begin{ex}[{\cite[Theorem 2.2]{barthel-may-riehl}}]\label{ex:chain} Let $\Ch_R$ denote the category of chain complexes of modules over a commutative ring $R$ with unit. We use the standard topological names for the following chain complexes. \[ (S^n)_k = \begin{cases} R & k = n \\ 0 & k \neq n \end{cases} \qquad (D^n)_k = \begin{cases} R & k = n,n-1 \\ 0 & \mathrm{else} \end{cases} \qquad I = \begin{cases} R & k=1 \\ R \oplus R & k = 0 \\ 0 & \mathrm{else} \end{cases}\] 
The non-trivial differential in $D^n$ is the identity; the non-trivial differential in $I$ is $x \mapsto (x,-x)$.

For any chain map $f \colon X \lra Y$, we can define ``mapping cylinder'' and ``mapping cocylinder'' factorizations 
\begin{equation}\label{eq:chain-factorization} X \lra Mf \lra Y \qquad X \lra Nf \lra Y \end{equation} where \[ \xymatrix{X \ar[r]^f \ar[d]_{i_0}  \ar@{}[dr]|(.8){\ulcorner} & Y \ar[d] & Nf \ar@{}[dr]|(.2){\lrcorner} \ar[r] \ar[d] & Y^I \ar[d]^{p_0} \\ X \otimes I \ar[r] & Mf & X \ar[r]_f & Y}\] are the \emph{mapping cylinder} and \emph{mapping cocylinder}; see \cite[Definition 18.1.1]{may-ponto}.

Both of the functorial factorizations \eqref{eq:chain-factorization} are algebraic weak factorization systems, i.e., the left factors define comonads, and the right factors define monads. Moreover, both algebraic weak factorization systems are \emph{enriched cofibrantly generated}, by which we mean that they are produced by the enriched version of the algebraic small object argument \cite[Theorem 13.2.1]{riehl-cathtpy}. The relevant enrichment is over the category of $R$-modules, which is why we require the ring $R$ to be commutative. The generating sets are \[ \I = \{ S^{n-1} \hookrightarrow D^n\}_{n \in \mathbb{Z}} \qquad \J = \{ 0 \lra D^n\}_{n \in \mathbb{Z}},\] respectively.
The underlying weak factorization systems of the algebraic weak factorization systems \eqref{eq:chain-factorization} define a model structure on $\Ch_R$, which in the literature is called the ``Hurewicz,''  ``relative,'' or ``absolute'' model structure, and which we call the \emph{Hurewicz model structure}. In this model structure weak equivalences are chain homotopy equivalences, fibrations are levelwise split epimorphisms and cofibrations are levelwise split monomorphisms.
\end{ex}

\begin{rmk} If ``cofibrant generation'' is meant in its classical (unenriched) sense, the sets $\I$ and $\J$ of Example \ref{ex:chain} generate the ``Quillen'', ``classical''  or ``projective'' model structure, where weak equivalences are quasi-isomorphisms and fibrations are epimorphisms. At least when $R=\mathbb{Z}$, the Hurewicz model structure is provably not cofibrantly generated \cite{christensen-hovey}.
\end{rmk}

{
Let $j$ and $f$ be morphisms in a locally small category $\D$. The set of commutative squares from $j$ to $f$ is defined by the following pullback. 
\[ \xymatrix{  \D(\mathrm{cod} j, \mathrm{dom} f) \ar@{-->}[dr] \ar@/^/[drr] \ar@/_/[ddr] \\ & \mathrm{Sq}(j,f) \ar[r] \ar[d] \ar@{}[dr]|(.2){\lrcorner} & \D(\mathrm{dom}\,j,\mathrm{dom}f) \ar[d] \\ & \D(\mathrm{cod}\,j,\mathrm{cod}f) \ar[r] & \D(\mathrm{dom}\,j,\mathrm{cod}f)}\] The map $\D(\mathrm{cod} j, \mathrm{dom} f) \to \mathrm{Sq}(j,f)$ carries a morphism from the codomain of $j$ to the domain of $f$ to the commutative square defined by precomposing with $j$ and postcomping with $f$. The map $j$ has the left lifting property with respect to $f$, which then has the right lifting property with respect to $j$, if and only if this map has a section.

If $\D$ is enriched in a complete and cocomplete closed symmetric monoidal category $\V$, this diagram can be interpreted in the category $\V$, in which case $\mathrm{Sq}(j,f)$ is the $\V$-object of commutative squares from $j$ to $f$. We say that $j$ \emph{has the enriched left lifting property with respect to $f$}, which then has the \emph{enriched right lifting property with respect to $j$}, if and only if this map has a section in $\V$. 

See \cite[\S 13.3]{riehl-cathtpy} for a discussion of the relationship between enriched and unenriched lifting properties.}

\subsection{The enriched algebraic small object argument}\label{ssec:enriched-SOA}

We now show that the enriched algebraic small object argument defines an accessible algebraic weak factorization system under a certain ``enriched smallness'' hypothesis. For convenience, let the base for enrichment be a complete and cocomplete closed symmetric monoidal category $\V$. 

\begin{thm}\label{thm:acc-awfs} Suppose $\J$ is a set (or small category) of arrows in a $\V$-cocomplete category $\D$.\footnote{If the underlying unenriched category of the $\V$-category $\D$ has all colimits and if $\D$ is tensored and cotensored over $\V$, then $\D$ is $\V$-cocomplete \cite[7.6.4]{riehl-cathtpy}. Conversely, any $\V$-cocomplete category admits $\V$-tensors and all unenriched colimits} Suppose further that $\D$ is locally presentable and that the representable functors $\D(d,-) \colon \D \lra \V$ preserve $\lambda$-filtered colimits for some $\lambda$, where $d$ ranges over the domains and codomains of objects of $\J$. Then the enriched algebraic small object argument defines an accessible algebraic weak factorization system whose underlying weak factorization system is enriched cofibrantly generated by $\J$.
\end{thm}
\begin{proof}
The enriched algebraic small object argument is described in \cite[13.2.1]{riehl-cathtpy}. It constructs the functorial factorization via an iterated colimit process. We require that $\D$ have ordinary colimits and tensors over $\V$ because these appear in the construction. It follows, in particular, that the functorial factorization is given by a pair of $\V$-functors $L,R \colon \D^\2 \lra \D^\2$.  It follows from \cite[13.4.2]{riehl-cathtpy} that the underlying weak factorization system of the algebraic weak factorization system $(\LL,\RR)$ is enriched cofibrantly generated by $\J$.

Let $E = \mathrm{cod}\cdot L = \mathrm{dom}\cdot R$. We will show that $E$ preserves $\lambda$-filtered colimits for some $\lambda$ larger than the index of accessiblity of each of the enriched representable functors. The functor $E$ is built from various colimit constructions, which commute with all colimits, and from functors $\mathrm{Sq}(j,-)\colon \D^2 \lra \V$ defined for each generating map $j$. For a morphism $f$ in $\D$,  $\mathrm{Sq}(j,f)$ is the ``object of commutative squares from $j$ to $f$'' defined via the pullback
\[ \xymatrix{ \mathrm{Sq}(j,f) \ar[r] \ar[d] \ar@{}[dr]|(.2){\lrcorner} & \D(\mathrm{dom}\,j,\mathrm{dom}f) \ar[d] \\ \D(\mathrm{cod}\,j,\mathrm{cod}f) \ar[r] & \D(\mathrm{dom}\,j,\mathrm{cod}f)}\] in $\V$. The domain and codomain functors preserve all colimits, because they are defined pointwise. The enriched representables $\D(\mathrm{dom}\, j,-)$ and $\D(\mathrm{cod}\,j,-)$ preserve $\lambda$-filtered colimits by hypothesis. Finally, the pullback, a finite limit, commutes with all filtered colimits.
\end{proof}

In order to apply Theorem \ref{thm:acc-awfs}, we supply conditions under which its central hypothesis --- the enriched smallness condition --- is satisfied. A particularly easy special case is when the ``underlying set'' functor $\V \lra \Set$ represented by the monoidal unit of $\V$ is \emph{conservative} (reflects isomorphisms) and accessible. This is the case, in particular, if $\V \lra \Set$ is monadic, and the monoidal unit is small (in the unenriched sense). Examples of categories $\V$ with this property include the categories of abelian groups, $R$-modules, vector spaces, compact Hausdorff spaces, and so on; in each of these cases, isomorphisms in $\V$ are structure-preserving bijections, and the unit object is finitely presented or compact.

{
\begin{lem} If $\V \lra \Set$ is conservative and accessible, then the enriched representable functors $\D(d,-) \colon \D \lra \V$ are accessible if and only if the unenriched representable functors $\D_0(d,-) \colon D \lra \Set$ are accessible.
\end{lem}
\begin{proof} Let $\D$ be a $\V$-category, and let $d \in \D$. The unenriched representable functor is the composite of the enriched representable functor and the underlying set functor \[ \D_0(d,-) \colon \D \xrightarrow{\D(d,-)} \V \xrightarrow{\ \ \ \ \ \ } \Set.\] 
If the enriched representable at $d$ is accessible, then $\D_0(d,-) \colon \D \lra \Set$ preserves $\lambda$-filtered colimits, where $\lambda$ is greater than the indexes of accessibility of the enriched representable and of $\V \lra \Set$. Conversely, if the unenriched representable preserves $\lambda$-filtered colimits, we have a bijection of underlying sets \[ \colim \D_0(d,x_\alpha) \xrightarrow{\ \ \ \cong \ \ \ } \D_0(d,\colim x_\alpha)\] for any $\lambda$-filtered diagram $(x_\alpha)$ in $\D$. The comparison morphism lifts to $\V$ and is an isomorphism there because $\V \lra \Set$ is conservative, which says exactly that the enriched representable $\D(d,-)$ preserves $\lambda$-filtered colimits.
\end{proof}

Frequently, a category $\D$ is a tensored, cotensored,  $\V$-enriched category because $\D$ is a closed monoidal category equipped with an adjunction \begin{equation}\label{eq:strong-monoidal} \xymatrix@C=4pc@R=4pc{ \V \ar@<1ex>[r]^-F \ar@{}[r]|-\perp & \D \ar@<1ex>[l]^-U}\end{equation} in which the left adjoint $F$ is \emph{strong monoidal}, i.e., if $\otimes$ denotes the monoidal product on $\V$ and $\wedge$ the monoidal product on $\D$,  there is a natural isomorphism $Fv \wedge Fv' \cong F(v \otimes v')$. If $\hom_\D$ denotes the closed structure on $\D$, then $\D$ becomes a $\V$-category with hom-objects $U\hom_\D(d,d')$, tensors $Fv \wedge d$, and cotensors $\hom_\D(Fv,d)$, for $d,d' \in \D$ and $v \in \V$; see \cite[3.7.11]{riehl-cathtpy}. For $\V = \sSet$ or $\sSet_*$, examples include all of the familiar enrichments (including $\D = \sSet_*$ over $\V=\sSet$). Another example is $\D = \Ch_R$ and $\V=\Mod_R$. The left adjoint is the inclusion in degree zero, and the right adjoint takes 0-cycles.

\begin{prop} If \eqref{eq:strong-monoidal} is a strong monoidal adjunction between locally presentable categories making $\D$ into a $\V$-category, then every object of $\D$ satisfies the enriched smallness condition.
\end{prop}
\begin{proof}
To say that $\V$ is locally $\lambda$-presentable implies that there exists a set $P$ of $\lambda$-presentable objects of $\V$ that form a dense generator. In particular, the unenriched representable functors $\V_0(p,-) \colon \V \lra \Set$ for $p \in P$ are \emph{jointly conservative} meaning that a map that induces a bijection after applying $\V_0(p,-)$ for all $p \in P$ is an isomorphism in $\V$ \cite[1.26]{adamek-rosicky}.

We will show that for each $d \in \D$ there is some regular cardinal $\kappa > \lambda$ so that $\D(d,-)\colon \D \lra \V$ preserves $\kappa$-filtered colimits. Here $\kappa$ should be chosen so that, for each $p \in P$, the unenriched representable $\D_0(Fp \wedge d,-) \colon \D \lra \Set$ preserves $\kappa$-filtered colimits. Because $\D$ is locally presentable, each unenriched representable preserves $\mu$-filtered colimits for some regular cardinal $\mu$. Because $P$ is a set, we can define $\kappa$ to be the supremum of such cardinals.

Consider a $\kappa$-filtered diagram $(x_\alpha)$ in $\D$. To show that the map \newline
$${\colim \D(d,x_\alpha) \xrightarrow{\ \ \cong \ \ } \D(d,\colim x_\alpha)}$$ is an isomorphism in $\V$, we apply $\V_0(p,-)$ for all $p \in P$. We have 
\begin{align*}
\V_0(p,\colim \D(d,x_\alpha)) &\cong \colim \V_0(p, \D(d, x_\alpha)) \\ \intertext{because $p$ is $\lambda$-small and hence $\kappa$-small}
&\cong \colim \D_0(Fp \wedge d, x_\alpha) \\ \intertext{by the definition of tensors in $\D$} &\cong \D_0(Fp \wedge d, \colim x_\alpha)\\ \intertext{because each $Fp \wedge d$ is $\kappa$-small} &\cong \V_0(p, \D(d,\colim x_\alpha)).\end{align*} Since the $p \in P$ form a dense generator for $\V$, this family of isomorphisms implies that $\colim \D(d,x_\alpha) \cong \D(d,\colim x_\alpha)$ in $\V$ as desired.
\end{proof}
}


\section{Spectral examples}\label{sec:spectra_ex}

In this section we apply Theorems \ref{thm:quillen-path} and \ref{thm:a-mod-cylinder} to categories of spectra. Since we require locally presentable categories, we work with the category $\cat{Sp}^\Sigma$ of symmetric spectra of simplicial sets. By \cite[III.4.13]{schwede} (see also Section 5 in \cite{hss}) there is a simplicial, combinatorial model structure on $\cat{Sp}^\Sigma$ with all objects cofibrant called the injective stable model structure.
This model structure can also be constructed using the general result about combinatorial model structures given in~\cite[3.6]{hess-shipley}. 

{ By \cite[Theorem 5.3.6, parts 3 and 5]{hss} $\SpS$, equipped with the injective stable model structure, is a $\V$--model category, where $\V=\SpS$, equipped with the projective stable model structure. Recall that the unit is cofibrant in $\V$. We can therefore apply Theorem~\ref{thm:a-mod-cylinder}, obtaining the following consequence.}

\begin{cor}\label{cor.of.2.2.3}
For any symmetric ring spectrum $A$, there exists an {\em injective model structure} on $\Mod_A$ left-induced from the injective stable model structure on $\SpS$ with cofibrations the injections and weak equivalences the stable equivalences.  
\end{cor}

The injective stable model structure on $\SpS$ is simplicial by~\cite[5.3.7]{hss}; see also~\cite[III.4.16]{schwede}.  It follows by Lemma 2.25 in~\cite{six-author} that this lifts to $\Mod_A$ as well.

\begin{prop}
The injective model structure on $\Mod_A$ is simplicial.  
\end{prop}

Let $A$ be a (strictly) commutative ring spectrum. Let $C$ be an $A$-coalgebra with comultiplication $C \to C \sm_A C$ and counit $C \to A$.
Consider the (forgetful, cofree)-adjunction between modules over $A$ and comodules over $C$.
\[ \xymatrix@C=4pc{ (\Mod_A)^{C} \ar@<1ex>[r]^-V \ar@{}[r]|-\perp & \Mod_A \ar@<1ex>[l]^-{-  \sm_A C}.}\]

\begin{thm}\label{thm:comod_spectra}
There exists a model structure on $(\Mod_A)^{C}$ left-induced by the adjunction above from the injective model structure on $\Mod_A$.
\end{thm}
\begin{proof}
This is an application of Theorem \ref{thm:quillen-path}.  First, note that tensoring with a simplicial set lifts to comodules.  Given a simplicial set $K$ and a $C$-comodule $M$ with coaction $\rho: M \to M \sm_A C$, then $K \otimes M = \Sigma^{\infty} K \sm M $ is a $C$-comodule with coaction 
$$ \Sigma^{\infty} K \sm M \xrightarrow{\Sigma^{\infty}K \sm \rho}  \Sigma^{\infty} K \sm (M \sm_A C) \cong (\Sigma^{\infty} K \sm M) \sm_A C. $$
There is a good cylinder object in $\sSet$ given by the factorization 
$$S^0 \coprod S^0 \rightarrowtail \Delta[1]_+ = I \xrightarrow{\sim} S^0.$$
The smash product of a comodule $M$ with this factorization in $\sSet$ lifts to $(\Mod_A)^{C}$.  Since $\Mod_A$ is simplicial and all objects are cofibrant, this defines the good cylinder objects in $ (\Mod_A)^{C}$ needed to apply Theorem \ref{thm:quillen-path}. 
\end{proof}

\begin{rmk}
The techniques we apply to construct certain model structures in the differential graded context in Section~\ref{section:dg_examples} do not easily extend to spectra, since in several places there we use the fact that coproducts and products agree in the differential graded context.
\end{rmk}


\section{DG examples}\label{section:dg_examples}

Let  $R$ be any commutative ring. In this section we apply the results of section \ref{sec:acyclicity}, particularly Theorems \ref{thm:quillen-path} and \ref{thm:square}, to establish the existence of 
\begin{itemize}
\item a  model category structure on the category $\cP\text{-}\Alg_{R}$ of differential graded (dg) $R$-algebras over any {$\Sigma$-split} dg operad $\cP$, where weak equivalences are underlying chain homotopy equivalences (Proposition \ref{prop:alg-mod-cat}); 
\medskip

\item  two  model category structures on the category $\cQ\text{-}\Coalg_{R}$ of dg $\Op Q$-coalgebras, for nice enough dg $R$-cooperads $\Op Q$ (in particular for the coassociative cooperad), where weak equivalences are either underlying chain homotopy equivalences or quasi-isomorphisms (Theorem \ref{thm:cooperadsQQ'});
\medskip

\item  two  model category structures on the category $\Ch_{R}^{C}$ of dg $C$-comodules over a coassociative dg $R$-coalgebra $C$, where weak equivalences are either underlying chain homotopy equivalences or quasi-isomorphisms (Corollary \ref{cor:comodule}); 
 \medskip

\item  two  model category structures on the category $(\cP, \cQ)\text{-}\Bialg_{R}$  of dg $(\cP, \cQ)$-bialgebras, for any  {$\Sigma$-split} dg $R$-operad $\cP$ and for a nice enough dg $R$-cooperad $\Op Q$ over $R$, where weak equivalences are underlying chain homotopy equivalences in both cases (Theorem \ref{thm:bialg});
\medskip

\item  two  model category structures on the category $\Alg_{R}^{H}$ of dg $H$-comodule algebras over an associative dg $R$-bialgebra $H$, where weak equivalences are underlying chain homotopy equivalences in both cases (Theorem \ref{thm:comod-alg}); and
\medskip

\item  two  model category structures on the category $\Mod_{A}^{C}$ of dg $C$-comodules over an $A$-coring $C$, where $A$ is an associative dg $R$-algebra, where weak equivalences are either underlying chain homotopy equivalences or quasi-isomorphisms (Theorem \ref{thm:a-coring}).
\end{itemize}

The model category structures from which we induce all of the structures listed above are  the Hurewicz model structure $(\Ch_R)_{\mathrm{Hur}}$  (Example \ref{ex:chain}) and the injective model structure $(\Ch_{R})_{\mathrm{inj}}$ \cite[Theorem 2.3.13]{hovey} on the category of unbounded chain complexes over $R$. Recall that in $(\Ch_R)_{\mathrm{Hur}}$ weak equivalences are homotopy equivalences, cofibrations are degreewise split monomorphisms, and fibrations are degreewise split epimorphisms, while in $(\Ch_{R})_{\mathrm{inj}}$ weak equivalences are quasi-isomorphisms,  and cofibrations are degreewise monomorphisms. The Hurewicz model category structure is enriched cofibrantly generated, as discussed in \S\ref{sec:enriched-cof-gen}. Hovey proved that $(\Ch_{R})_{\mathrm{inj}}$ is cofibrantly generated, in the usual sense of the word. With respect to the usual tensor product and $\hom$ of chain complexes,  $(\Ch_{R})_{\mathrm{Hur}}$ is a monoidal model category \cite[Theorem 1.15]{barthel-may-riehl}.

Throughout this section $\otimes$ denotes the usual tensor product on $\Ch _{R}$, i.e., tensoring over $R$.

\subsection {Preliminaries}
We begin this section with a brief overview of those notions concerning chain complexes and their algebraic structures that are necessary to stating and proving our results below.

\subsubsection{Cylinders}\label{sec:cyl}

The following functorial  construction  of a good cylinder in $(\Ch_{R})_{\mathrm{Hur}}$, mentioned briefly in Example \ref{ex:chain}, proves very useful for establishing the existence of the model category structures below.  Recall the definition of the interval complex $I$ from Example \ref{ex:chain}, and let $t$ denote a generator of $I_{1}$ and $\del_{0}t$, $\del _{1}t$ generators of $I_{0}$ so that $dt=\del_{0}t-\del_{1}t$.  For any chain complex $X$, let $\Cyl(X)=X\otimes I$, 
$$i:X\oplus X \lra \Cyl (X):(x_{0},x_{1}) \mapsto x_{0}\otimes \del_{0}t + x_{1}\otimes \del_{1}t,$$
and 
$$q:\Cyl (X) \lra X: \begin{cases} x\otimes t &\mapsto 0\\ x\otimes \del_{0}t &\mapsto x\\ x \otimes \del_{1}t&\mapsto x.\end{cases}.$$
It is clear that $i$ and $q$ are, respectivly, a cofibration and a weak equivalence in both $(\Ch_{R})_{\mathrm{inj}}$ and $(\Ch_{R})_{\mathrm{Hur}}$.  Indeed, 
$$i_{k}:X \lra \Cyl(X): x\mapsto x\otimes \del_{k}t$$
is a chain homotopy inverse to $q$ for $k=1,2$.

\subsubsection{Operads and cooperads}\label{sssec:operad}

We recall here briefly those elements of the theory of (co)operads and their (co)algebras that are necessary to understanding the results below. We refer the reader to \cite{loday-vallette} for further details.

To begin we remind the reader of  three different and useful monoidal structures on any cocomplete symmetric monoidal category $(\V, \otimes, \uI)$. Given symmetric sequences $\Op{X}$ and $\Op{Y}$ of objects in $\V$, their \emph{levelwise tensor product},  $\Op X\otimes \Op Y$, is the symmetric sequence defined in arity $n$ by
$$(\Op X\otimes \Op Y)(n)= \Op X(n) \otimes \Op Y(n),$$ 
enowed with the diagonal $\Sigma_{n}$-action, while  their \emph{graded tensor product}, $\Op{X} \odot \Op{Y}$, is given by
\[
   ( \Op{X} \odot \Op{Y}) (n) = \coprod_{m = 0}^{n} \left(\Op{X}\left(m\right) \otimes \Op{Y}\left(n-m\right)\right) \underset{\Sigma_{m} \times \Sigma_{n-m}}{\otimes} R[\Sigma_{n}],
\]
and  their  \emph{composition product}, $\Op X \circ \Op Y$,  by
\[
    \Op{X} \circ \Op{Y} = \coprod_{m \geq 0} \Op{X}(m)
    \underset{\Sigma_{m}}{\otimes} \Op{Y}^{\odot m}.
\]
An \emph{$\V$-operad} is a symmetric sequence in $\V$ endowed with the structure of a monoid with respect to the composition product. {A $\V$-operad $\cP$ is \emph{$\Sigma$-split} if it is a retract of the levelwise tensor product of operads $\cP \otimes \Op A$, where $\Op A$ denotes the \emph{associative operad}, given by $\Op A(n)=\uI [\Sigma_{n}]= \coprod_{\Sigma_{n}}\uI$ for all $n$.} 

Given a $\V$-operad $\cP$, a \emph{$\cP$-algebra} is an object $A$ in $\V$ equipped with a family of morphisms in $\V$
$$\cP (n)\otimes _{\Sigma_{n}} A^{\otimes n} \lra A$$
satisfying appropriate associativity and unitality conditions.  We denote the category of $\cP$-algebras by $\cP\text{-}\Alg_{\V}$ or by $\cP\text{-}\Alg_{R}$, when $\V=\Ch_{R}$.  There is an adjunction
$$\adjunct {\V}{\Op P\text{-}\Alg_{\V}}{F_{\Op P}}{U_{\Op P}},$$
where $F_{\Op P}$ is the \emph{free $\Op P$-algebra} functor, defined by 
$$F_{\Op P}(X)= \coprod_{n\geq 0}\Op P(n)\otimes_{\Sigma_{n}}X^{\otimes n}.$$
In the case $\V=\Ch_{R}$, note that {if $\Op P$ is $\Sigma$-split, then $F_{\Op P}$ preserves chain homotopy equivalences, since for every $X$, $F_{\Op P}(X)$ is a retract of 
$$F_{\Op P\otimes \Op A}(X)\cong \coprod_{n\geq 0}\Op P(n)\otimes X^{\otimes n}.$$
Note that the isomorphism above is built from $\Sigma_{n}$-equivariant isomorphisms 
$$\cP(n)\otimes_{R} R[\Sigma_{n}] \xrightarrow \cong \cP(n)\otimes_{R} R[\Sigma_{n}]\colon p\otimes \sigma \mapsto p\cdot \sigma^{-1}\otimes \sigma,$$
where the source is equipped with the diagonal $\Sigma_{n}$-action and the target with the free $\Sigma_{n}$-action.}  In particular, if $X$ is contractible, then so is $F_{\Op P}(X)$ {whenever $\Op P$ is $\Sigma$-split}.

Note that if $\V$ is locally presentable, then so is $\cP\text{-}\Alg_{\V}$, by \cite[2.78]{adamek-rosicky}.

The dual case is somewhat more delicate to set up. Let $\Op J$ denote  the unit for the composition product, i.e.,  the symmetric sequence in $\V$ that is $\uI$ in arity $1$ and the 0-object in all other arities. A \emph{$\V$-cooperad} consists of a symmetric sequence $\Op{Q}$ in $\V$, together with a counit $\varepsilon : \Op{Q} \lra \Op{J}$ and a family of chain maps
\begin{equation}\label{eqn:coop-struct}
    \psi_{\vec{n}} : \Op{Q}(n) \lra \Op{Q}(m) \otimes
    \Op{Q}(n_{1}) \otimes \cdots \otimes \Op{Q}(n_{m})
\end{equation}
for all sequences $\vec n=(n_{1}, \ldots, n_{m})$ of non-negative integers with $\sum_{j} n_{j} = n$ and for all $n$, satisfying coassociativity, counitality, and equivariance properties.

If $\Op{Q}(0)$ is the 0-object, then the family of morphisms (\ref{eqn:coop-struct}) yields a sequence of $\Sigma_{n}$-equivariant morphisms,
\[
    \tilde{\psi}_{n} : \Op{Q}(n) \lra \bigoplus_{m = 1}^{n} \left(
        \Op{Q}(m) \otimes \Op{Q}^{\odot m}(n)\right)^{\Sigma_{m}}
\]
and thus a morphism of symmetric sequences, $\psi : \Op{Q} \lra \Op{Q} \circ \Op{Q}$,
obtained by composing with the natural map from fixed points to orbits.
By~\cite[Proposition 1.1.15]{fresse:99}, the family of morphisms (\ref{eqn:coop-struct}) can  be recovered from $\psi$.

Associated to any symmetric sequence $\Op X$ in $\V$, there is an endofunctor 
$$\Gamma_{\Op X}: \V\lra \V,$$ 
which is defined on objects by
\[
	\Gamma_{\Op{X}}(C) = \coprod_{n \geq 0} \left( \Op{X}(n) \otimes C^{\otimes n} \right)^{\Sigma_{n}}.
\]
If $\Op{Q}$ is a cooperad such that $\Op Q(0)$ is the $0$-object, then the functor $\Gamma_{\Op{Q}}$ is the underlying functor of a comonad on $\V$. An \emph{$\Op{Q}$-{coalgebra}} is  a coalgebra for this comonad. 
More explicitly, a $\Op{Q}$-coalgebra consists of a chain complex $C$ together with structure maps
$$\rho_{n}:C \lra \Op{Q}(n) \otimes C^{\otimes n}$$ 
for all $n>0$ that are appropriately counital and coassociative.  
Only finitely  many of the $\rho_{n}$ are non-trivial, and thus all $\Op Q$-coalgebras are \emph{conilpotent}.  In a slight abuse of  notation, we write $\Gamma_{\Op Q}$ for the cofree $\Op Q$-coalgebra functor.

By \cite[Proposition A.1]{ching-riehl}, the category $\cQ\text{-}\Coalg_{\V}$ of $\cQ$-coalgebras in $\V$ is locally presentable if there is some  regular cardinal $\kappa$ such that $\V$ is locally $\kappa$-presentable, and $\Gamma_{\cQ}$ preserves $\kappa$-filtered colimits.

It is easy to see that the levelwise tensor product lifts to the category of dg-cooperads.

\subsubsection{Bar and cobar constructions}

Fix a commutative ring $R$. Let $\cP$ be a dg operad and $\cQ$ a dg cooperad over $R$.  A \emph{twisting morphism}  $\tau \colon \cQ \lra \cP$ encodes the information of an operad morphism $\bold\Omega \cQ \lra \cP$ or, equivalently, of a cooperad morphism $\cQ \lra \mathbf{Bar}\cP$, where $\mathbf{Bar}$ and $\bold\Omega$ denote the operadic bar and cobar functors, for the definitions and properties of which we refer the reader to \cite{loday-vallette}.    A twisting morphism $\tau\colon\cQ \lra \cP$ induces an adjunction 
$$\adjunct {\cQ\text{-}\Coalg^{\mathrm{conil}}}{\cP\text{-}\Alg,}{\Om_{\tau}}{\Bar_{\tau}}$$
where $\cQ\text{-}\Coalg^{\mathrm{conil}}$ denotes the category of conilpotent $\cQ$-coalgebras.

Let $\Coalg^{\mathrm{conil}, \eta}_{R}$ and $\Alg^{\ve}_{R}$ denote the categories of conilpotent, coaugmented dg $R$-coalgebras and  of augmented dg $R$-algebras, respectively. The classical cobar-bar adjunction
\begin{equation}\label{eqn:bar-cobar}
\adjunct{\Coalg^{\mathrm{conil}, \eta}_{R}}{\Alg^{\ve}_{R}}{\Omega}{\Bar},
\end{equation}
which is the special case of the adjunction above for the usual twisting morphism between the coassociative cooperad and the associative operad, is a key tool for understanding and comparing certain model category structures that we construct in this section.   We recall here the definition of these functors.

Let $T$ denote the endofunctor on the category of  graded $R$-modules given by
$$TX=\oplus _{n\geq 0}X^{\otimes n},$$
where $X^{\otimes 0}=R$.  An element of the summand $X^{\otimes n}$ of $TX$ is a sum of terms denoted $x_{1}|\cdots |x_{n}$, where $x_{i}\in V$ for all $i$. Note that $T$ underlies both the free associative algebra functor, also denoted $T$, with multiplication given by concatenation,  and the cofree coassociative coalgebra functor $T^{\mathrm{co}}$, with comultiplication given by taking sums of all possible splittings of a tensor words.

The \emph {suspension} endofunctor $s$ on the category of graded $R$-modules is defined on objects $X=\bigoplus _{i\in \mathbb Z} X_ i$ by
$(sX)_ i \cong X_ {i-1}$.  Given a homogeneous element $x$ in
$X$, we write $sx$ for the corresponding element of $sX$. The suspension $s$ admits an obvious inverse, which we denote $\si$.

The \emph{bar construction} functor $\Bar$ is defined by 
$$\Bar A=\left(T^{\mathrm{co}} (s\overline A), d_{\Bar}\right)$$
where $\overline A$ denotes the augmentation ideal of $A$, and if $d$ is the differential on $A$, then
$$\pi\circ d_{\Bar}(sa)=-s(da)$$
and
$$\pi\circ d_{\Bar}(sa|sb)=(-1)^{|a|}s(ab),$$
where $\pi: T(s\overline A)\lra s\overline A$ is the projection. The entire differential is determined by its projection onto $s\overline A$, since the graded $R$-module underlying $\Bar A$ is naturally a cofree coassociative coalgebra, with comultiplication given by splitting of words. 
 
 The \emph{cobar construction} functor $\Om $ is defined by 
$$\Om C= \left(T (\si \overline C), d_{\Om}\right)$$
where $\overline C$ denotes the coaugmentation coideal of $C$, and if $d$ denotes the differential on $C$ and $c$ is a homogeneous element of $C$, then
$$d_{\Om}(\si c)=-s(dc) + (-1)^{|c_{i}|} \si c_{i}|\si c^{i},$$ 
where the reduced comultiplication applied to $c$ is $c_{i}\otimes c_{}{}^{i}$ (using Einstein implicit summation notation).  The entire differential is determined by its restriction to $\si \overline C$, since the graded $R$-module underlying $\Om C$ is naturally a free associative algebra. 

For further details of this adjunction, we refer the reader to \cite{husemoller-moore-stasheff} or \cite {neisendorfer}.

\subsection{Algebras over operads}\label{sec:alg-r}
Let $R$ be any commutative ring, and let $\Op P$ be an operad in $\Ch_{R}$.  It follows from \cite[Proposition 4.1]{berger-moerdijk-operads} that if $\Op P$ is \emph{$\Sigma$-split}, then the category $\Op P\text{-}\Alg_{R}$ of $\Op P$-algebras admits a model category structure right-induced from the projective model category structure on $\Ch_{R}$ by the forgetful functor $U\colon \Op P\text{-}\Alg_{R}\lra \Ch_{R}$.  

We prove here an {analogous result for the Hurewicz model structure and provide a new proof in the case of the projective model category structure.}

\begin{prop}\label{prop:alg-mod-cat} Let $R$ be any commutative ring, and let $\Op P$ be {a $\Sigma$-split} operad in $\Ch_{R}$.  The forgetful functor $U\colon \Op P\text{-}\Alg_{R}\lra \Ch_{R}$ creates a right-induced model category structure on $\cP\text{-}\Alg_{R}$, when $\Ch_{R}$ is endowed with the Hurewicz model structure or the projective model structure.
\end{prop}

\begin{proof} By the observations above and Theorem \ref{thm:right-awfs}, we can right-induce both weak factorization systems, so it remains only to establish the acyclicity condition,  i.e., that $\llp\,{U^{-1}\Fib_{\mathrm{Hur}}} \subset U^{-1}\WE_{\mathrm{Hur}}$.
 
Given $i \colon A \lra B$ with the left lifting property against $U^{-1}\Fib_{\mathrm{Hur}}$, consider the following factorization
\[ \xymatrix@R=4pc@C=4pc{ A \ar[d]_i \ar[r]^-{\in U^{-1}\WE_{\mathrm{Hur}}} & A \coprod \big (F_{\Op P}(B \oplus s^{-1}B,D)\big) \ar[d]^{i+q\in U^{-1}\Fib_{\mathrm{Hur}}} \\ B \ar@{=}[r] \ar@{-->}[ur] & B}\] 
where $F_{\Op P}$ is the free $\Op P$-algebra functor, $\coprod$ denotes the coproduct of $\Op P$-algebras,  $D(b)= s^{-1}b$ for all $b\in B$, and $q$ is specified by $q(b)=b$ and $q(s^{-1}b)=db$ for all $b\in B$, where $d$ is the differential on $B$. The first map is obviously an underlying homotopy equivalence, since $F_{\Op P}(B \oplus s^{-1}B, D)$ is contractible, and $i + q$ is an underlying Hurewicz fibration, since it is split as a map of graded $R$-modules. Applying 2-of-6, we conclude that $i \in U^{-1}\WE_{\mathrm{Hur}}$ and therefore the right-induced model structure on $\Op P\text{-}\Alg_R$ exists.

Note that the same argument works if we start with the projective model structure on $\Ch_{R}$, since the first map in the factorization above is a quasi-isomorphism and the second one is an epimorphism.
\end{proof}

We denote this model structure $(\Op P\text{-}\Alg_{R})_{\mathrm{Hur}}$. We refer the reader to Proposition \ref{prop:fib-cofib} for the construction of an explicit cofibrant replacement functor when $\Op P =\Op A$, the associative operad.

\begin{rmk} The proof above does not work for $(\Ch_{R})_{\mathrm{inj}}$, as the map $i+q$ is usually not an injective fibration.  We do not know if it is possible to replace $(\Ch_{R})_{\mathrm {Hur}}$ by $(\Ch_{R})_{\mathrm{inj}}$ in the statement of the proposition above, for any operad $\cP$.
\end{rmk}

\subsection{Coalgebras over cooperads}\label{sec:coalg}
There is a long history of efforts to establish model category structures on categories of coalgebras over cooperads.  Quillen established the first model category structure on a particular category of coalgebras over a comonad,  the category of 1-connected, cocommutative dg coalgebras over $\mathbb Q$ in \cite {quillen}. Thirty years later Getzler and Goerss proved the existence of a model category structure on the category of dg coalgebras over a field in an unpublished manuscript \cite{getzler-goerss}. Around the same time Blanc provided conditions under which a ``right'' model category structure could be transfered from an underlying model category to a category of coalgebras  \cite [Theorem 7.6]{blanc}.  Hinich then generalized Quillen's work,  defining a simplicial model category structure on the category of unbounded cocommutative dg-coalgebras over a field of characteristic zero, where the class of weak equivalences consists of those dg-coalgebra maps that induce a quasi-isomorphism on the associated Chevalley-Eilenberg complexes  \cite {hinich}. 

In 2003 Aubry and Chataur proved the existence of  model category structures on categories of  certain cooperads and coalgebras over them, when the underlying category is that of  unbounded chain complexes over a field \cite{aubry-chataur}. Smith established results along the same lines in \cite{smith} in 2011.  In 2010, Stanculescu applied the dual of the Quillen path-object argument to establish a model structure on comonoids in nice enough monoidal model category, given a functorial cylinder object for comonoids  \cite{stanculescu}.    A general existence result for model structure on categories of coalgebras over a given comonad on a model category was proved in \cite{hess-shipley}.

Most recently Drummond-Cole and Hirsh applied the main theorem of \cite{six-author} to establish the existence of a poset of model category structures on categories of coalgebras over coaugmented, weight-graded cooperads in the category of chain complexes, either unbounded or bounded below, over a field, where the weak equivalences and cofibrations are created by the cobar construction associated to a twisting morphism \cite{cole-hirsh}.  Their work generalizes results of Vallette for Koszul cooperads \cite{vallette}.  Finally, in \cite{yalin} Yalin proved the existence of cofibrantly generated, left-induced model category structure on the category of dg-coalgebras over a reduced operad in the category of finite-dimensional vector spaces over a field.

Here we obtain as a consequence of Theorem \ref{thm:cooperad-cylinder} that for any commutative ring $R$ and any  nice enough dg-cooperad $\Op Q$,  the category of differential graded $\Op Q$-coalgebras admits a model category structure.

\begin{thm}\label{thm:cooperadsQQ'} Let $R$ be a commutative ring.  Let  $\Op Q$ be a cooperad on $\M$ such that the category $\cQ\text{-}\Coalg_{R}$ is locally presentable.  If there is a cooperad $\Op Q'$ in $\Ch_{R}$ equipped with a map $\Op Q \otimes \Op Q' \lra \Op Q$ of cooperads and such that $R$ admits the structure of a $\Op Q'$-coalgebra, extending to 
\begin{equation}\label{eq:unit-dgcylinder}  R\oplus R  \rightarrowtail I \xrightarrow{\sim} R\end{equation}
in $\cQ'\text{-}\Coalg_{R}$, where cofibrations and weak equivalences are created in $\Ch_{R}$,  then there exists model structures on the category $\cQ\text{-}\Coalg_{R}$ of  $\cQ$-coalgebras that are left-induced via the forgetful functor $V$
\[ \xymatrix@R=4pc@C=4pc{ \cQ\text{-}\Coalg_{R} \ar@<1ex>[r]^-V \ar@{}[r]|-\perp & \Ch_R \ar@<1ex>[l]^-{\Gamma_\cQ}}\] 
from $(\Ch_{R})_{\mathrm{Hur}}$ and from $(\Ch_{R})_{\mathrm{inj}}$.
\end{thm}

We denote the two model category structures of the theorem above $(\cQ\text{-}\Coalg_{R})_{\mathrm{Hur}}$ and $(\cQ\text{-}\Coalg_{R})_{\mathrm{inj}}$.

\begin{proof}  This is a special case of Theorem \ref{thm:cooperad-cylinder}.  Note first that every object in both $(\Ch_{R})_{\mathrm{Hur}}$ and $(\Ch_{R})_{\mathrm{inj}}$ is cofibrant, so the condition on underlying-cofibrant replacements is trivially satisfied.   In the case of the Hurewicz structure, take $\M=\V= (\Ch_{R})_{\mathrm{Hur}}$ and $\boxtimes=\otimes$, and in the case of the injective structure, one can take $\M =(\Ch_{R})_{\mathrm{inj}}$, $\V= (\Ch_{R})_{\mathrm{proj}}$, and $\boxtimes=\otimes$ by~\cite[3.3]{shipley-HZ}
or $\M =(\Ch_{R})_{\mathrm{inj}}$, $\V= (\Ch_{R})_{\mathrm{Hur}}$, and $\boxtimes=\otimes$.  Note, each of these pairs satisfies the hypotheses of Theorem~\ref{thm:cooperad-cylinder}, but for the last pair, $\M$ is not a $\V$-model category.
\end{proof}

\begin{rmk} A $\cQ'$-coalgebra structure on the tensor unit $R$ corresponds to coherent choice of coaugmentation (or basepoint) in every arity of $\cQ'$. In other words, if $\cQ'$ is a cooperad in coaugmented chain complexes, then $R$ has a natural $\cQ'$-coalgebra structure.  
\end{rmk}

\begin{rmk}\label{rmk:Hopf-coop} If  $\cQ$ is Hopf cooperad, i.e., a cooperad in $\Alg_{R}$, then the multiplicative structure provides a map of cooperads $\cQ \otimes \cQ \lra \cQ$, and $\cQ$ is naturally coaugmented.
 It follows in this case that if the interval $I$ admits a $\cQ$-coalgebra structure, extending the natural $\cQ$-coalgebra structure on $R\oplus R$, then Theorem \ref{thm:cooperadsQQ'} holds for $\cQ$-coalgebras.
\end{rmk}

The special case of the category $\Coalg_{R}$ of {counital} coassociative coalgebras is worthy of separate mention.

\begin{cor}\label{cor:comonoids} Let $R$ be a commutative ring. There exist model structures on $\Coalg_R$ left-induced from $(\Ch_{R})_{\mathrm{Hur}}$ and from $(\Ch_{R})_{\mathrm{inj}}$ by the forgetful functor.
\end{cor}

\begin{proof}  Let $\Op A^{\text{co}} $ denote the coassociative cooperad.  For every chain complex $X$, $\Gamma_{\Op A^{\text{co}} }(X)=\bigoplus_{n\geq 0} X^{\otimes n}$, whence $\Gamma_{\Op A^{\text{co}} }$ preserves filtered colimits and thus $\Coalg_R$ is locally presentable.  The coassociative cooperad is a Hopf cooperad, and the interval complex $I$ admits an obvious coassociative comultiplication, extending the trivial comultiplication on either endpoint.  By Remark \ref{rmk:Hopf-coop} we can conclude.\end{proof}

We denote the two model category structures of the corollary above $(\Coalg_{R})_{\mathrm{Hur}}$ and $(\Coalg_{R})_{\mathrm{inj}}$. We refer the reader to Proposition \ref{prop:fib-cofib} for the construction of an explicit fibrant replacement for coaugmented coassociative coalgebras.

\begin{rmk} We conjecture that the results of Drummond-Cole and Hirsh \cite{cole-hirsh} mentioned above can be generalized to show that there is a poset of model category structures on $\cQ\text{-}\Coalg_{R}$ determined by the twisting morphisms with source $\cQ$, at least for nice enough cooperads $\cQ$.
\end{rmk}

Let $C$ be any dg $R$-coalgebra. In \cite[Theorem 6.2]{hess-shipley} it was shown that  $(\Ch^{+}_R)^C$, the category of non-negatively graded dg $C$-comodules, admits a model category structure left-induced from an injective model category structure on $\Ch^{+}_R$ for which the weak equivalences are quasi-isomorphisms, as long as $R$ is semi-hereditary, and $C$ is non-negatively graded, $R$-free, and of finite type.   Other existence results for model structures for categories of comodules over a field were recently established by Drummond-Cole and Hirsh in \cite {cole-hirsh}, along the lines of earlier work of Lef\`evre-Hasegawa \cite{lefevre}.

Here we establish a more general existence result, for any commutative ring $R$ and any dg $R$-coalgebra 
$C$, as a special case of Theorem \ref{thm:cooperadsQQ'}.

\begin{cor}\label{cor:comodule} Let $R$ be any commutative ring and $C$ any dg $R$-coalgebra.
There exist  model category structures on the category $\Ch_R^C$ of right $C$-comodules that are left-induced  along the forgetful functor $V\colon \Ch_R^C \lra \Ch_{R}$, with respect to both the Hurewicz model structure and the injective model structure.
\end{cor}  

\begin{proof} Let $\cQ_{C}$ be the cooperad that is $C$ in arity 1 and $0$ in all other arities, so that the category of $\cQ$-coalgebras is isomorphic to the category of $C$-comodules.  It is obvious that there is a morphism of cooperads $\cQ_{C}\otimes \Op A^{\text{co}} \to \cQ_{C}$.  Moreover the cofree coalgebra functor, which is simply $-\otimes C$, preserves all colimits and thus, in particular, filtered colimits, which implies that $\Ch_R^C$ is locally presentable. As mentioned in the proof of the previous corollary, the interval complex $I$ admits an obvious coassociative comultiplication, extending the trivial comultiplication on either endpoint, so we can apply Theorem \ref{thm:cooperadsQQ'} to conclude.
\end{proof}

We denote the model category structures $(\Ch_R^C)_{\mathrm{Hur}}$ and $(\Ch_R^C)_{\mathrm{inj}}$.

\subsection{Bialgebras}
Let $\Op P$ be any dg operad, and let $\Op Q$ be a nice enough cooperad, in the sense of our results above.
In this section we establish the existence of two Quillen-equivalent model category structures on the category $(\Op P, \Op Q)\text{-}\Bialg_{R}$ of dg $(\Op P, \Op Q)$-bialgebras over $R$, for any commutative ring $R$. We show that these structures are distinct if $R$ is not of characteristic 2, and $\cP$ and $\cQ$ are the associative operad and the coassociative cooperad.  To the best of the authors' knowledge, the only other example of model category structure on $(\Op P, \Op Q)\text{-}\Bialg_{R}$ in the literature is due to Yalin \cite{yalin}, who considered the case of positively graded dg-$(\Op P, \Op Q)$-bialgebras  over a field $\Bbbk$ of characteristic zero, where $\Op P$ and $\Op Q$ are reduced operads in finite-dimensional $\Bbbk$-vector spaces.  Starting from the left-induced model category structure on the category of $\Op Q$-coalgebras mentioned in the introduction to section \ref{sec:coalg}, he proved that  the forgetful functor from $(\Op P, \Op Q)$-bialgebras to $\Op Q$-coalgebras right-induced the desired model category structure.  

\subsubsection{Model structures on $(\Op P, \Op Q)\text{-}\Bialg_R$} 
Recall the right-induced model structure $(\Op P\text{-}\Alg_{R})_{\mathrm{Hur}}$  of Proposition \ref{prop:alg-mod-cat}   and the left-induced model structure $(\cQ\text{-}\Coalg_{R})_{\mathrm{Hur}}$ of Theorem \ref{thm:cooperadsQQ'}.

\begin{thm}\label{thm:bialg} Let $R$ be a commutative ring.  Let $\Op P$ be a {$\Sigma$-split} operad in $\Ch_{R}$, and let $\cQ$ and $\cQ'$ be cooperads in $\Ch_R$ satsifying the hypotheses of Theorem \ref{thm:cooperadsQQ'}.   Let $\mathbb P$ and $\mathbb Q$ denote the free $\Op P$-algebra monad and cofree $\Op Q$-coalgebra comonad on $\Ch_{R}$.  If there is a distributive law $\mathbb P\mathbb Q \Rightarrow \mathbb Q\mathbb P$, then the category $(\Op P, \Op Q)\text{-}\Bialg_{R}$ of $(\Op P, \Op Q)$-bialgebras admits two Quillen-equivalent model category structures,  one of which is left-induced from $(\Op P\text{-}\Alg_{R})_{\mathrm{Hur}}$, while the other is right-induced from $(\Op Q\text{-}\Coalg_{R})_{\mathrm{Hur}}$.  In particular, the weak equivalences in both cases are the bialgebra morphisms such that the underlying chain map is a chain homotopy equivalence.  
\end{thm}

In particular this theorem applies to the category of ordinary bialgebras  $\Bialg_{R}$, where $\Op P$ is the associative operad and $\Op Q$ is the coassociative cooperad. We show in the next section (Proposition \ref{prop:distinct}) that the two model structures of the theorem above on $\Bialg_{R}$ are actually distinct when the characteristic of $R$ is different from 2.

\begin{proof} We apply Corollary \ref{cor:monad_comonad} to the square
\[ \xymatrix@R=4pc@C=4pc{ \Ch_{R} \ar@{}[r]|{\perp} \ar@<-1ex>[d]_{F_{\Op P}} \ar@<-1ex>[r]_{\Gamma_{\Op Q}} & \Op Q\text{-}\Coalg_{R}  \ar@<1ex>[d]^{\widehat F_{\Op P}} \ar@<-1ex>[l]_V \\  
\ar@{}[r]|{\top} \ar@{}[u]|{\dashv} \Op P\text{-}\Alg_{R} \ar@<1ex>[r]^{\widehat\Gamma_{\Op Q}} \ar@<-1ex>[u]_U & (\Op P, \Op Q)\text{-}\Bialg_{R} \ar@<1ex>[u]^U \ar@<1ex>[l]^V\ar@{}[u]|{\vdash} }\] 
using the existence results for Hurewicz model category structures on $\Op P\text{-}\Alg_{R}$ and $\Op Q\text{-}\Coalg_{R}$ from sections \ref{sec:alg-r} and \ref{sec:coalg}. Here $F_{\Op P}$ denotes the free  $\Op P$-algebra functor and $\Gamma_{\Op Q}$ the cofree $\Op Q$-coalgebra functor, while $\widehat F_{\Op P}$ and $\widehat \Gamma_{\Op Q}$ denote their lifts.   Note that the lifts exist and the chain complex underlying $\widehat F_{\Op P}(X, \delta)$ is $F_{\Op P}X$, thanks to the distributive law $\mathbb P\mathbb Q \Rightarrow \mathbb Q\mathbb P$.  Note moreover that $(\Op P, \Op Q)\text{-}\Bialg_{R}$ is locally presentable, since $ \Op Q\text{-}\Coalg_{R}$ is locally presentable by hypothesis.

Since all the categories in the diagram are locally presentable, and $UV=VU$ and  $F_{\Op P}V\cong V\widehat F_{\Op P}$, we conclude that the two desired model structures on the category $(\Op P, \Op Q)\text{-} \Bialg_R$ exist and are Quillen equivalent.
\end{proof}

We denote the two model structures on $(\Op P, \Op Q)\text{-}\Bialg_{R}$ given by the theorem above by $\big((\Op P, \Op Q)\text{-}\Bialg_{R}\big)_{\urcorner}$ and $\big((\Op P, \Op Q)\text{-}\Bialg_{R}\big)_{\llcorner}$, i.e., $\big((\Op P, \Op Q)\text{-}\Bialg_{R}\big)_{\urcorner}$ is constructed by left-inducing from $(\Ch_{R})_{\mathrm{Hur}}$ to $(\cQ\text{-}\Coalg_{R})_{\mathrm{Hur}}$ then by right-inducing along the righthand vertical adjunction, while $\big((\Op P, \Op Q)\text{-}\Bialg_{R}\big)_{\llcorner}$ is constructed by right-inducing from $(\Ch_{R})_{\mathrm{Hur}}$ to $(\cP\text{-}\Alg_{R})_{\mathrm{Hur}}$ then by left-inducing along the bottom horizontal adjunction.

\subsubsection{Fibrant and cofibrant replacements in $\Bialg_R$}\label{sec:cofib-fib-rep}

In this section we construct explicit fibrant replacements in $(\Bialg_{R})_{\urcorner}$ and cofibrant replacements in $(\Bialg_{R})_{\llcorner}$, in terms of the cobar and bar constructions, see (\ref{eqn:bar-cobar}).
The desired replacements in $\Bialg_{R}$ arise from the following replacements in $(\Alg_{R})_{\mathrm {Hur}}$ and $(\Coalg_{R})_{\mathrm {Hur}}$, which are of independent interest.

Observe first that it follows from the definition of $(\Ch_{R})_{\mathrm{Hur}}$ that fibrations in $(\Alg_{R})_{\mathrm{Hur}}$ are degreewise surjective algebra morphisms that are split as morphisms of graded $R$-modules, while cofibrations in $(\Coalg_{R})_{\mathrm{Hur}}$ are degreewise injective coalgebra morphisms that are split as morphisms of graded $R$-modules.
Recall also that a dg $R$-algebra is \emph{augmented} if equipped with a morphism $A \lra R$ of algebras, while a dg $R$-coalgebra is \emph{coaugmented} if equipped with a morphism $R \lra C$ of coalgebras.

\begin{prop}\label{prop:fib-cofib}  Let $R$ be a commutative ring.
\begin{enumerate}
\item For any augmented dg $R$-algebra $A$, the counit  of the cobar-bar adjunction
$$\varepsilon_{A}\colon \Omega \Bar A \lra A$$
is a cofibrant replacement in $(\Alg_{R})_{\mathrm{Hur}}$.  
\item For any coaugmented dg $R$-coalgebra $C$, the unit of the cobar-bar adjunction
$$\eta_{C}\colon C \lra \Bar \Omega C$$
is a fibrant replacement in $(\Coalg_{R})_{\mathrm{Hur}}$.
\end{enumerate}
\end{prop}

Existence of the desired replacements for bialgebras now follows easily.

\begin{cor}\label{cor:(co)fib-rep}  Let $R$ be a commutative ring, and let $H$ be a conilpotent dg $R$-bialgebra.
\begin{enumerate}
\item The counit  of the cobar-bar adjunction
$$\varepsilon_{H}\colon \Omega \Bar H \lra H$$
is a cofibrant replacement in $(\Bialg_{R})_{\llcorner}$.  
\item The unit of the cobar-bar adjunction
$$\eta_{H}\colon H \lra \Bar \Omega H$$
is a fibrant replacement in $(\Bialg_{R})_{\urcorner}$.
\end{enumerate}
\end{cor}

\begin{proof} [Proof of Corollary \ref{cor:(co)fib-rep}] Note that any bialgebra is naturally augmented (by its counit) as an algebra and coaugmented (by its unit) as a coalgebra. By \cite[Theorem 3.12]{hess:twisting}, the algebra $\Omega \Bar H$ admits a natural bialgebra stucture with respect to which $\varepsilon_{H}$ is a morphism of bialgebras. Dually  the coalgebra $\Bar\Omega  H$ admits a natural bialgebra stucture with respect to which $\eta_{H}$ is a morphism of bialgebras.  To conclude it suffices therefore to recall that cofibrant objects and weak equivalences in $(\Bialg_{R})_{\llcorner}$ are created in $(\Alg_{R})_{\mathrm{Hur}}$, while fibrant objects and weak equivalences in $(\Bialg_{R})_{\urcorner}$ are created in $(\Coalg_{R})_{\mathrm{Hur}}$, and then to apply Proposition \ref{prop:fib-cofib}.
\end{proof}

To prove Proposition \ref{prop:fib-cofib}, we need to know that the chain maps underlying the unit and counit  of the cobar-bar adjunction are chain homotopy equivalences.

\begin{lem}\label{lem:alg-htpy-eq}\cite[Theorem 4.4]{husemoller-moore-stasheff} The chain map underlying the counit map $\varepsilon_{A}\colon \Omega \Bar A \lra A$ is a chain homotopy equivalence for every augmented dg $R$-algebra $A$.
\end{lem}

\begin{lem}\label{lem:coalg-htpy-eq}\cite[Theorem 4.5]{husemoller-moore-stasheff} The chain map underlying the unit map $\eta_{C}\colon C\lra \Bar \Omega C$ is a chain homotopy equivalence for every coaugmented, conilpotent dg $R$-coalgebra $C$.
\end{lem}

The final elements we need to prove Proposition \ref{prop:fib-cofib} concern conditions under which $\Bar A$ is fibrant and $\Om C$ is cofibrant, formulated in the following terms.

\begin{defn}  Let $R$ be a commutative ring.
\begin {enumerate}
\item A dg $R$-algebra $A$ is \emph{split-nilpotent} if there is a sequence 
$$\cdots \xrightarrow{p_{n+1}} A[n]\xrightarrow {p_{n}} A[n-1]\xrightarrow {p_{n-1}} \cdots \xrightarrow {p_{1}}A[0]\xrightarrow {p_{0}} A[-1]=R$$
of fibrations in $(\Alg_{R})_{\mathrm{Hur}}$ such that each $p_{n}$ is a morphism of augmented algebras, $A=\lim_{n}A[n]$, and $\ker p_{n}$ is a trivial (non-unital) algebra {(i.e., a square-zero ideal)} for all $n$.
\item A dg $R$-coalgebra $C$ is \emph{split-conilpotent} if there is a sequence 
$$R=C[-1]\xrightarrow {j_{0}}C[0]\xrightarrow {j_{1}}\cdots \xrightarrow{j_{n-1}} C[n-1]\xrightarrow {j_{n}} C[n]\xrightarrow {j_{n+1}} \cdots $$
of cofibrations in $(\Coalg_{R})_{\mathrm{Hur}}$ such that each $j_{n}$ is a morphism of coaugmented coalgebras, $C=\operatorname{colim}_{n}C_{n}$, and $\coker j_{n}$ is a trivial (non-counital) coalgebra for all $n$.
\end{enumerate}
\end{defn}

\begin{rmk}\label{rmk:split} Over a field any nilpotent algebra is split-nilpotent, and similarly for coalgebras.  {Over an arbitrary commutative ring $R$, any dg $R$-coalgebra with cofree underlying graded coalgebra is split-conilpotent. Moreover if the underlying graded algebra of dg $R$-algebra $A$ is free on a graded module concentrated in strictly positive degrees, then $A$  is split-nilpotent.  In particular, $\Bar A$ is split-conilpotent for every augmented dg $R$-algebra $A$, and $\Omega C$ is split-nilpotent for every coaugmented dg $R$-coalgebra $C$ such that $C_{1}=0$.}  

To see this, recall that $T$ denotes the endofunctor of graded $R$-modules underlying both the free algebra functor and the cofree coalgebra functor. Write  
$$T^{\leq n}X=\bigoplus _{0\leq k \leq n} X^{\otimes k}\quad\text{and}\quad T^{\geq n}X=\bigoplus _{ k \geq n} X^{\otimes k}.$$

Let $(TX, d)$ be a dg $R$-algebra with free underlying graded algebra. {If $X$ is concentrated in strictly positive degrees, the tower of fibrations in $(\Alg_{R})_{\mathrm{Hur}}$ converging to $(TX, d)$ is the sequence of $R$-split quotient maps
$$ \cdots \twoheadrightarrow TX /T^{\geq 3}X \twoheadrightarrow TX /T^{\geq 2}X \twoheadrightarrow R$$
with the obvious induced differential and multiplication in each layer.  It converges to $TX$,  since  in each degree the tower consists only of isomorphisms after a finite stage, due to the hypothesis on $X$.} On the other hand, for a dg $R$-coalgebra $(T^{\mathrm{co}}X, d)$ with cofree underlying coalgebra, the associated sequence of cofibrations in $(\Coalg_{R})_{\mathrm{Hur}}$ is the sequence of $R$-split inclusions
and
$$R \hookrightarrow T^{\leq 1}X \hookrightarrow T^{\leq 2}X \hookrightarrow\cdots$$
with the obvious restricted differential and comultiplication in each layer.
\end{rmk}

\begin{lem}\label{lem:split-fib} Let $R$ be a commutative ring.
\begin{enumerate}
\item If $A$ is a split-nilpotent dg $R$-algebra, then $\Bar A$ is a fibrant object of $(\Coalg_{R})_{\mathrm{Hur}}$.
\item If $C$ is a split-conilpotent dg $R$-coalgebra, the $\Omega C$ is a cofibrant object of $(\Alg_{R})_{\mathrm{Hur}}$.
\end{enumerate}
\end{lem}

\begin{proof}  We prove (2) explicitly and leave the dual proof of (1) to the reader.
Let  $C$ be a coaugmented, split-conilpotent dg $R$-coalgebra, with comultiplication $\Delta$ and with associated sequence of cofibrations
$$R=C[-1]\xrightarrow {j_{0}}C[0]\xrightarrow {j_{1}}C[1] \xrightarrow {j_{2}}\cdots \xrightarrow{j_{n-1}} C[n-1]\xrightarrow {j_{n}} C[n]\xrightarrow {j_{n+1}} \cdots ,$$
giving rise to a filtration
$$R \subseteq \Om C[0]\subseteq \Om C[1] \subseteq \cdots \subseteq \Om C[n-1] \subseteq \Om C[n]\subseteq \cdots$$
of $\Om C$.  For any $n\geq 1$ and $c\in C[n]$, let $\bar c\in \coker j_{n}$ denote its class in the quotient. Note that since $j_{n}$ is $R$-split, $C[n]\cong C[n-1] \oplus \coker j_{n}$ as graded $R$-modules.  Moreover, since the induced comultiplication on the quotient coalgebra $\coker j_{n}$ is trivial,  
$$\overline \Delta (\bar c)= \Delta (\bar c) - \bar c \otimes 1- 1\otimes \bar c\in C[n-1]\otimes C[n-1]$$ 
for all $\bar c \in \coker j_{n}$.

By hypothesis $\Omega C[0] = T(s^{-1} \overline C[0], s^{-1} ds)$, which is cofibrant in $(\Alg_{R})_{\mathrm{Hur}}$. Higher stages of the filtration of $\Om C$ can be built up inductively from pushouts in $\Alg_{R}$, given in stage $n$ by
\[ \xymatrix{ T(s^{-2} \overline{\coker j_{n} }, s^{-2}ds^{2}) \ar[d] \ar[r] \ar@{}[dr]|(0.8){\ulcorner} & \Omega C[n-1] \ar[d] \\ T(s^{-2}\overline{\coker j_{n}} \oplus s^{-1}\overline{\coker j_{n}},D) \ar[r] & \Omega C[n].}\] 
Here, for all $\bar c\in \coker j_{n}$, the upper horizontal map sends $s^{-2}\bar c$ to $ s^{-1}c_i | s^{-1}c^i$, where $\overline \Delta(\bar c)=c_{i}\otimes c^{i}$ (using the Einstein summation convention), while the differential $D$ extends $s^{-2}ds^{2}$ and is given by $Ds^{-1}\bar c= s^{-2}\bar c-s^{-1}\overline {dc}$ . In particular, we can inductively build up the cobar construction as the colimit of pushouts of algebra morphisms given by applying $T$ to cofibrations in $(\Ch_{R})_{\mathrm{Hur}}$, whence $\Omega C$ is cofibrant in $\Alg_R$.
\end{proof}

\begin{proof}[Proof of Proposition \ref{prop:fib-cofib}] To prove (1), apply Lemma \ref{lem:alg-htpy-eq}, Remark \ref{rmk:split}, and Lemma \ref{lem:split-fib}(2). Similarly,  (2) follows from Lemma \ref{lem:coalg-htpy-eq}, Remark \ref{rmk:split}, and Lemma \ref{lem:split-fib}(1). 
\end{proof}

As a consequence of Corollary \ref{cor:(co)fib-rep}, we can show that  the two model structures we have constructed on $\Bialg_{R}$ are distinct, at least when $R$ is of characteristic different from 2.  It is probably possible to modify the example below to cover the case of characteristic 2 as well.

\begin{prop}\label{prop:distinct} If $R$ is a commutative ring of characteristic different from $2$, then the model structures   $(\Bialg_{R})_{\urcorner}$ and $(\Bialg_{R})_{\llcorner}$ are distinct.  In particular, there are bialgebras that are cofibrant in $(\Bialg_{R})_{\llcorner}$ but not in $(\Bialg_{R})_{\urcorner}$
\end{prop}

\begin{proof} Fix  an even number $m\geq 2$, and let  $H$ denote the bialgebra that is cofree as a coalgebra on a free graded $R$-module $X$ with exactly one generator $x$, of degree $m$.  Note that for degree reasons the differential is trivial.  The multiplication on $H$ is chosen to be trivial as well.  We claim that $\Omega \Bar H$, which is cofibrant in $(\Bialg_{R})_{\llcorner}$ by Corollary \ref{cor:(co)fib-rep}, is not cofibrant in $(\Bialg_{R})_{\urcorner}$.

Let $\widehat H$ denote the bialgebra that is cofree as a coalgebra on a free graded $R$-module $\widehat X$ with exactly three generators $x$, $y$, and $z$, of degrees $m$, $4m$ and $4m+1$, respectively.  The differential $d$ sends $x$ and $y$ to 0 and $z$ to $y$.   The multiplication $\widehat \mu$ on $\widehat H$ is specified by requiring that the composite
$$\widehat H\otimes\widehat H \xrightarrow {\widehat \mu} \widehat H \twoheadrightarrow \widehat X$$
be zero except on $x|x \otimes x|x$, which is sent to $y$.  The coalgebra map $p\colon \widehat H\lra H$ that is induced by the obvious projection $\widehat X \lra X$ clearly respects both the differential and the multiplication.  It is an acyclic fibration in $(\Bialg_{R})_{\urcorner}$, as the underlying coalgebra map is obtained by applying the cofree coalgebra functor $T^{\mathrm{co}}$ to an acyclic fibration in $\Ch_{R}$.

To prove that $\Om \Bar H$ is not cofibrant in  $(\Bialg_{R})_{\urcorner}$, we show that there is no morphism of bialgebras lifting the counit $\ve_{H}\colon \Om \Bar H \lra H$ through the acyclic fibration $p\colon \widehat H \lra H$.  If  $\widehat \ve\colon \Om \Bar H \lra \widehat H$ is a morphism of dg $R$-algebras such that $p\widehat \ve= \ve_{H}$, then it necessarily satisfies
$$\widehat \ve (s^{-1}sx)=x \quad\text{and}\quad \widehat \ve \big(s^{-1}s(x|x)\big)=x|x,$$
for degree reasons.  Again for degree reasons, there is some $a\in R$ such that 
$$\Delta \big(s^{-1}s(x|x)\big)= s^{-1}s(x|x) \otimes 1  + a\cdot ( s^{-1}sx\otimes s^{-1}sx) + 1 \otimes s^{-1}s(x|x),$$
where $\Delta$ is the comultiplication on $\Om \Bar H$.  Since $\Delta$ is an algebra morphism,
{\small\begin{align*}
\Delta \big(s^{-1}s(x|x)|s^{-1}s(x|x)\big)=&s^{-1}s(x|x)|s^{-1}s(x|x)\otimes 1 + 1 \otimes s^{-1}s(x|x)|s^{-1}s(x|x)\\
& +2 s^{-1}s(x|x)\otimes s^{-1}s(x|x)\\
&+ a\cdot \big(s^{-1}s(x|x) |s^{-1}sx\otimes s^{-1}sx + s^{-1}sx\otimes s^{-1}s(x|x) |s^{-1}sx\\
&\quad\qquad+s^{-1}sx|s^{-1}s(x|x)\otimes s^{-1}sx+s^{-1}sx\otimes s^{-1}sx|s^{-1}s(x|x)\big)\\
&+a^{2}\cdot (s^{-1}sx|s^{-1}sx\otimes s^{-1}sx| s^{-1}sx),
\end{align*}}
and therefore, since $\widehat \ve$ is  a morphism of algebras,
{\small\begin{align*}
(\widehat \ve\otimes \widehat \ve)\Delta \big(s^{-1}s(x|x)|s^{-1}s(x|x)\big)=&(x|x)\cdot (x|x)\otimes 1 + 1 \otimes (x|x)\cdot (x|x)+2 (x|x)\otimes (x|x)\\
&+ a\cdot \big((x|x) \cdot x\otimes x + x\otimes (x|x) \cdot x\\
&\quad\qquad+ x\cdot (x|x)\otimes x+ x\otimes x\cdot (x|x)\big)\\
&+a^{2}\cdot( x\cdot x\otimes x\cdot x)\\
=& y\otimes 1 + 1 \otimes y+2 (x|x)\otimes (x|x),
\end{align*}}
where the last equality follows from the definition of the multiplication in $\widehat H$.  On the other hand, if $\widehat \Delta$ denotes the comultiplication on $\widehat H$, then
$$\widehat \Delta \widehat\ve\big(s^{-1}s(x|x)|s^{-1}s(x|x)\big)=\Delta (y)=y\otimes 1 + 1 \otimes y,$$
whence $\widehat \ve$ is not a morphism of coalgebras if $2\not =0$.  It follows that a lift of $\ve_{H}$ through $p$ cannot be simultaneously a morphism of algebras and a morphism of coalgebras if the characteristic of $R$ is different from 2.
\end{proof}

\begin{rmk}  We conjecture that for any twisting morphism $\tau\colon \cQ \lra \cP$ such that the unit and counit of the induced adjunction 
$$\adjunct{\cQ\text{-}\Coalg}{\cP\text{-}\Alg}{\Om_{\tau}}{\Bar_{\tau}}$$
are objectwise chain homotopy equivalences, analogues of  Proposition \ref{prop:fib-cofib} and Corollary \ref{cor:(co)fib-rep} hold.  Moreover it is likely that a counter-example similar to that above can be constructed in this context, establishing that 
$\big((\Op P, \Op Q)\text{-}\Bialg_{R}\big)_{\urcorner}$ and $\big((\Op P, \Op Q)\text{-}\Bialg_{R}\big)_{\llcorner}$ are indeed distinct.  Proving analogues of  Corollary \ref{cor:(co)fib-rep} and of the counter-example above would require a generalization of  \cite[Theorem 3.12]{hess:twisting}, i.e., that if $H$ is a $(\cP, \cQ)$-bialgebra, then $\Om_{\tau}\Bar_{\tau} H$ and $\Bar _{\tau}\Om_{\tau} H$ both admit natural $(\cP, \cQ)$-bialgebra structures.  We suspect that this is the case, at least under reasonable conditions on the twisting morphism $\tau$, but the proof is beyond the scope of this article. 
\end{rmk}

\subsection{Comodule algebras}

Let $H$ be a bimonoid in $\Ch_{R}$, and consider $\Alg_R^H$, the category of $H$-comodules in the category of $\Alg_R$ of differential graded $R$-algebras.  In \cite[Theorem 3.8]{six-author} it was shown that if $R$ is a field, and $H$ is of finite type and non-negatively graded, then the category of non-negatively graded $H$-comodule algebras $(\Alg^{+}_R)^H$ admits a model category structure left-induced from the model category structure on $\Alg^{+}_{R}$ that is right-induced from the projective structure on $\Ch_{R}$.  Here we generalize this result to any commutative ring and any bimonoid $H$, at the price of working with chain homotopy equivalences rather than quasi-isomorphisms as our weak equivalences.

There is  a commutative diagram of forgetful functors $(UV=VU)$ admitting adjoints
\[ \xymatrix@R=4pc@C=4pc{ \Ch_R \ar@{}[r]|{\perp} \ar@<-1ex>[d]_T \ar@<-1ex>[r]_{-\otimes H} & \Ch_R^H  \ar@<1ex>[d]^T \ar@<-1ex>[l]_V \\  
\ar@{}[r]|{\top} \ar@{}[u]|{\dashv} \Alg_R \ar@<1ex>[r]^{- \otimes H} \ar@<-1ex>[u]_U & \Alg_R^H. \ar@<1ex>[u]^U \ar@<1ex>[l]^V\ar@{}[u]|{\vdash} }\] 
As recalled in Example \ref{ex:chain}, the Hurewicz model structure is enriched cofibrantly generated and thus accessible, and all four categories are locally presentable. 

The following theorem is now an immediate consequence of Corollary \ref{cor:monad_comonad}, since there is a distributive law 
$$\chi:T\circ (-\otimes H) \Longrightarrow (-\otimes H) \circ T,$$
with components given by
$$\chi_{X}:T(X\otimes H) \lra (TX) \otimes H: (x_{1}\otimes h_{1})|\cdots | (x_{n}\otimes h_{n}) \mapsto (x_{1}|\cdots |x_{n}) \otimes h_{1}\cdots h_{n}.$$

\begin{thm}\label{thm:comod-alg}There exist  right- and  left-induced model structures on $\Alg_R^H$, created by $U:\Alg_R^H \lra \Ch^{H}$ and  $V: \Alg_R^H\lra \Alg_{R}$ repectively, with respect to the model structures $(\Ch_R^H)_{\mathrm{Hur}}$ and $(\Alg_{R})_{\mathrm{Hur}}$. In particular, the identity defines a left Quillen functor from the right-induced model structure to the left-induced one:
\[ \xymatrix@C=4pc@R=4pc{  (\Alg_R^H)_{\mathrm{right}} \ar@<1ex>[r]^-\id \ar@{}[r]|-\perp & (\Alg_R^H)_{\mathrm{left}}. \ar@<1ex>[l]^-\id}\]
\end{thm}

\begin{rmk}  We suspect that the two model category structures of the theorem above are different in general, but do not have a specific counter-example.
\end{rmk}

\subsection{Coring comodules}\label{subsec.dg.coring}

Let $A \in \Alg_R$, and let  $\Coring_A$ denote the category of $A$-corings, i.e., comonoids in the monoidal category of $(A,A)$-bimodules, where the monoidal product is given by tensoring over $A$. For any $A$-coring $C$ there are adjoint functors
\[ \xymatrix@R=4pc@C=4pc{ \Ch_R \ar@<1ex>[r]^{- \otimes A} \ar@{}[r]|{\perp} & \Mod_A \ar@<1ex>[l]^U \ar@<-1ex>[d]_{-\otimes_A C} \ar@{}[d]|{\vdash} \\ & (\Mod_A)^C \ar@<-1ex>[u]_V}\] Here $\Mod_A$ is the category of right $A$-modules, and $(\Mod_A)^C$ is the category of $C$-comodules in right $A$-modules. 

In \cite[Theorem 6.2]{hess-shipley} it was shown that  $(\Mod^{+}_A)^C$, the category of non-negatively graded dg $C$-comodules in $A$-modules, admits a model category structure left-induced from an injective model category structure on $\Mod^{+}_A$ for which the weak equivalences are quasi-isomorphisms, as long as $R$ is semi-hereditary, $A$ is simply connected, and $C$ is $A$-semifree of finite type.   Here we prove a much more general existence result, for any commutative ring $R$, any dg $R$-algebra $A$, and any dg $A$-coring $C$.

Note that in this case it is not possible to reverse the order of the adjunctions, so there is no square diagram as before. Note also that all three categories are locally presentable.

We endow  $\Mod_A$ either with the $r$-module structure of \cite{barthel-may-riehl}, which is right-induced from the Hurewicz model structure on $\Ch_R$ along $U$, or with the injective model category structure of \cite[Proposition 3.11]{hess-shipley} (see also Theorem \ref{thm:a-mod-cylinder}), in which the weak equivalences are the quasi-isomorphisms, and the cofibrations are the levelwise injections.  We want to left-induce these model structures from $\Mod_A$ to $(\Mod_A)^C$ along $V$. By Theorem \ref{thm:right-awfs} we can transfer both weak factorization systems, so it remains only to prove the acyclicity condition, i.e., $\rlp{(V^{-1}\Cof)}\subset V^{-1}\WE$, where $\Cof$ and $\WE$ denote the cofibrations and weak equivalences in either the $r$-model structure or the injective model structure on $\Mod_{A}$. 

For lifting the injective structure the usual cylinder construction works well, see below. 
For establishing acyclicity for the $r$-model structure we need a cofibrant replacement functor. For this we use the two-sided bar construction and refer the reader to \cite[Appendix]{gugenheim-may} for an introduction to this construction, of which the important properties for us include the following. 

\begin{prop}\label{prop:Bar_properties} The functor $\Bar (-,A,A)\colon \Mod_{A}\lra \Mod_{A}$
\begin{enumerate}
\item   lifts to an endofunctor on $(\Mod_A)^C$ compatible with the natural augmentation $\varepsilon \colon \Bar (-,A,A) \lra \id$, and
\item  preserves colimits. 
\end{enumerate}
\end{prop}

\begin{proof}
(1) Let $\rho\colon X \lra X\otimes _{A}C$ denote the $C$-coaction on $A$.  For any $x\in X$, write $\rho(x)= x_i \otimes v^i$ using the Einstein summation convention, i.e., we sum over any index that is both a subscript and a superscript, as $i$ is here.

For any element $x \otimes sa_1 |sa_2 | ... |sa_n \otimes b$ of $\Bar(X,A,A)$, we define the coaction map 
$$\widehat \rho: \Bar(X,A,A) \lra \Bar(X,A,A)\otimes_A C \cong \Bar(X,A,C)$$ 
by
$$x \otimes sa_1 |sa_2 | ... |sa_n \otimes b \mapsto \begin{cases} x_i \otimes v^ib &: n=0, \\
0 &: n>0.  \end{cases}$$
Straightforward computations show that $\widehat \rho$ commutes with the differentials, is a map of right $A$ modules, is natural in $(X,\rho)$, and makes the natural augmentation 
$$\varepsilon \colon  \Bar(X,A,A) \lra X: x \otimes sa_1 |sa_2 | ... |sa_n \otimes b \mapsto \begin{cases} xb &: n=0 \\
0 &: n>0 \end{cases}$$
into a map of comodules.

(2) Since colimits in $\Mod_{A}$ are created in $\Ch_{R}$, it suffices to observe that tensor product with a fixed object and direct sums preserve colimits in $\Ch_{R}$. 
\end{proof}

\begin{cor} Let $C$ be any $A$-coring.  When equipped with cofibrations and weak equivalences created in either the $r$-model structure or the injective model structure on  $\Mod_{A}$, the category $(\Mod_A)^C$ admits a cofibrant replacement functor $Q: (\Mod_A)^C\lra (\Mod_A)^C$, specified on objects by
$Q(X,\rho)=\big(\Bar (X,A,A), \widehat \rho \big)$.
\end{cor}

\begin{proof} Recall that the augmentation $\varepsilon _{X}\colon\Bar (X,A,A)\lra X$ is a chain homotopy equivalence, and therefore a quasi-isomorphism, for all $A$-modules $X$, since it arises as the realization of an augmented simplicial object with extra degeneracies. By Proposition \ref{prop:Bar_properties} and its proof, it remains only to establish cofibrancy of $\Bar (X,A,A)$, which is immediate in the case of the injective structure and  is the content of \cite[Proposition 10.18]{barthel-may-riehl} in the case of the $r$-model structure.
\end{proof}

Observe finally that $(\Mod_A)^C$ admits good cylinder objects with respect to both the $r$-model structure and the injective structure, given by tensoring with $R\oplus R \rightarrowtail I \xrightarrow \sim R$.  The next theorem is therefore an immediate consequence of Theorem \ref{thm:quillen-path}.

\begin{thm}\label{thm:a-coring} Let $R$ be any commutative ring. For any $A\in \Alg_{R}$ and any $A$-coring $C$, the category $(\Mod_A)^C$ of $C$-comodules in $A$-modules admits  model category structures left-induced from the $r$-model structure and from the injective model structure on $\Mod_{A}$ via the forgetful functor.
\end{thm}

\section{The Reedy model structure as a model category of bialgebras}\label{sec:diagram_ex}

{Recall from \S\ref{ssec:injective} that for} any small category $\D$ and bicomplete category $\M$, the forgetful functor $\M^\D \lra \M^{\ob\D}$ is both monadic and comonadic, with left and right adjoints given by left and right Kan extension. If $\M$ is a model category, then $\M^{\ob\D} \cong \prod_{\ob\D}\M$ inherits a pointwise-defined model structure, with weak equivalences, cofibrations, and fibrations all defined pointwise in $\M$. The right-induced model structure on $\M^\D$ is called the \emph{projective model structure}, while the left-induced model structure is called the \emph{injective model structure}.

A \emph{Reedy category} is a small category $\D$ equipped with a direct subcategory $\D^+ \subset \D$ of ``degree-increasing'' morphisms and an inverse subcategory $\D^- \subset \D$ of ``degree-decreasing'' morphisms. In \S\ref{ssec:reedy}, we show that if $\D$ is a Reedy category, then the category $\M^\D$ of Reedy diagrams is the category of bialgebras for a monad and a comonad on $\M^{\ob\D}$ induced respectively by the forgetful functors $\M^{\D^+} \lra \M^{\ob\D}$ and $\M^{\D^-} \lra \M^{\ob\D}$. Moreover, the Reedy model structure can be understood as the model structure that is left-induced from the projective model structure on diagrams indexed by the directed category $\D^+$, or equally as the model structure that is right-induced from the injective model structure on the inverse category $\D^-$. In this way, the Reedy model structure on $\M^\D$ can be recovered from Theorem \ref{thm:square} in the case where $\M$ is an accessible model category.

{In \S\ref{ssec:generalized}, we extend these observations to generalized Reedy categories. In \S\ref{ssec:exact}, we observe that our proofs imply that the inclusion of these subcategories into a Reedy category $\D$ defines an \emph{exact square}. Exact squares are an essential ingredient in the theory of derivators, a very general framework in which to study homotopy limits and colimits. While several general classes of exact squares are known, none appear to accommodate this particular Reedy category example, which for that reason we suspect will be of interest.}

\subsection{Reedy diagrams as bialgebras}\label{ssec:reedy}

A \emph{direct category} is a small category $\D$ that can be equipped with a degree function $d \colon \ob\D \lra \mathbb{N}$ such that every non-identity morphism of $\D$ ``raises the degree,'' in the sense that the degree of its codomain is strictly greater than the degree of its domain. {The dual notion is that of an  \emph{inverse category}, i.e., a small category whose opposite is a direct category.   Homotopy-theoretic interest in direct categories stems from the following classical result; see, e.g., \cite[5.1.3]{hovey}.

\begin{prop} If $\D$ is a direct category, and $\M$ is a model category, then $\M^\D$ admits the right-induced model structure from the pointwise model structure on $\M^{\ob\D}$. Dually, if $\D$ is an inverse category, and $\M$ is a model category, then $\M^\D$ admits the left-induced model structure from the pointwise model structure on $\M^{\ob\D}$.
\end{prop}}

In the right-induced model structures on diagrams indexed by a direct category, the fibrations and weak equivalences are defined objectwise, and in the left-induced model structures on diagrams indexed by an inverse category, the cofibrations and weak equivalences are defined objectwise. A Reedy category generalizes these notions. We recall the following definition, as formulated in \cite{hovey}.

\begin{defn} A \emph{Reedy category} is a small category $\R$ together with a \emph{degree function} $d: \ob \R \lra \mathbb{N}$ and two wide subcategories $\R^+$ and $\R^{-}$ satisfying the following axioms.
\begin{enumerate}
\item Non-identity morphisms of $\R^+$ {strictly} raise degree.
\item  Non-identity morphisms of $\R^{-}$ {strictly} lower degree.
\item Every morphism in $\R$ factors uniquely as a morphism in $\R^{-}$ followed by a morphism in $\R^+$.
\end{enumerate}
\end{defn}

{
\begin{ex} The categories $\Delta_+$ and $\Delta$ of finite, respectively finite non-empty, ordinals are Reedy categories. Opposites of Reedy categories and finite products of Reedy categories are again Reedy categories.
\end{ex}}

\begin{rmk} An inverse or a direct category is an example of a Reedy category. In the former case, all morphisms are degree-decreasing, and in the latter case all morphisms are degree-increasing. Conversely, if $\R$ is a Reedy category, then the subcategories $\R^+$ and $\R^-$ are respectively direct and inverse.
\end{rmk}

The following constructions play an essential role in the definition of the Reedy model structure.

\begin{defn} Let $\R$ be a Reedy category, let $r \in \ob\R$, and let $\C$ be a bicomplete category. 
The \emph{$r$-th latching object} of $\Phi \in \C^{\R}$ is 
$$L_r\Phi = \colim (\R^+_{<\mathrm{deg}(r)}/r \xrightarrow U\R \xrightarrow \Phi \C ),$$
and the  \emph{$r$-th matching object}  is 
$$M_r\Phi = \lim (r/\R^{-}_{<\mathrm{deg}(r)} \xrightarrow U \R \xrightarrow\Phi \C ),$$
where $\R^+_{<\mathrm{deg}(r)}/r$ is the slice category over $r$ with objects restricted to degree less than the degree of $r$ and morphisms in $\R^+$, $r/\R^{-}_{<\mathrm{deg}(r)}$ is defined dually, and $U$ denotes a forgetful functor from a slice category.
\end{defn}

\begin{rmk}\label{rmk:relative-matching} For every $r\in \ob \R $ and every morphism $\tau\colon  \Phi \lra \Psi \in \C^\R$, there are natural morphisms in $\C$: the \emph{relative latching map} 
$$\ell_{r}(\tau)\colon\Phi(r) \sqcup_{L_r\Phi} L_r\Psi \lra \Psi(r)$$
and the  \emph{relative matching map}
$$m_{r}(\tau)\colon\Phi (r) \lra M_r\Phi \times_{M_r\Psi} \Psi(r).$$
\end{rmk}

The following result, due to Kan based on work of  Reedy, defines the \emph{Reedy model structure.}

\begin{thm}[{\cite[Theorem 5.2.5]{hovey}}] Let $(\M, \Fib, \Cof, \WE)$ be a model category and $\R$  a Reedy category. Then $\M^{\R}$ can be equipped with a model structure  where a morphism $\tau\colon \Phi \lra \Psi$ is 
\begin{itemize}
\item a weak equivalence if and only if $\tau_r\in \WE$ for every $r\in \ob \R$;
\item a cofibration if and only if $\ell_{r}(\tau)\in \Cof$ for every $r\in \ob \R$; and
\item a fibration if and only if $m_{r}(\tau)\in \Fib$ for every $r\in \ob \R$.
\end{itemize}
\end{thm}

We now explain how Reedy diagrams arise as bialgebras defined by combining diagrams indexed by the direct and inverse subcategories. We have a diagram
\[ \xymatrix@C=4pc@R=4pc{ \M^{\ob \R} \ar@{}[r]|{\perp} \ar@<-1ex>[d]_L \ar@<-1ex>[r]_{R} & \M^{\R^-}  \ar@<1ex>[d]^L \ar@<-1ex>[l]_V \\  \ar@{}[r]|{\top} \ar@{}[u]|{\dashv} \M^{\R^+} \ar@<1ex>[r]^{R} \ar@<-1ex>[u]_U & \M^{\R} \ar@<1ex>[u]^U \ar@<1ex>[l]^V\ar@{}[u]|{\vdash} }\] 
where $U,V$ are restriction functors, $L$ denotes the left Kan extension, and $R$ denotes the right Kan extension along the respective inclusions 
\[ \xymatrix{ \ob \R \ar@{^(->}[r] \ar@{^(->}[d] & \R^- \ar@{^(->}[d] \\ \R^+ \ar@{^(->}[r] & \R}\]

\begin{lem}\label{lem:reedy-exact} $LV \cong VL$ and $RU \cong UR$.
\end{lem}
\begin{proof}
The two assertions are equivalent, and dual. Given a diagram $\Phi \in \M^{\R^-}$, $LV(\Phi) \in \M^{\R+}$ is the diagram defined at $r \in \R^+$ by
\[ LV(\Phi)(r) := \coprod_{x, \R^+(x,r)} \Phi (x).\] Employing the coequalizer formula for pointwise left Kan extensions, we see that this is isomorphic to $VL(\Phi)$ on account of the following coequalizer diagram
\[ \xymatrix{ \coprod_{x,y, \R^-(x,y) \times \R(y,r)} \Phi(x) \ar@<.5ex>[r] \ar@<-.5ex>[r] & \coprod_{x,\R(x,r)} \Phi(x) \ar[r] & \coprod_{x, \R^+(z,r)} \Phi (z),}\] where the quotient map is defined on the component indexed by a morphism $x \lra r$ by taking the Reedy factorization $x \twoheadrightarrow z \rightarrowtail r$, applying $\Phi$ to the left factor, and injecting into the component indexed by the right factor.
\end{proof}

The key observation relating Reedy diagrams to the framework of section \ref{section:adjoint_squares} is the following.

\begin{lem}\label{lem:reedy-bialgebra}  If $\R$ is a Reedy category, and $\M$ is any bicomplete category, then  $\M^\R$ is the category of bialgebras for the monad $\TT$ and comonad $\KK$ on $\M^{\ob\R}$ induced by the adjunctions
\[ \xymatrix{ \M^{\R^+} \ar@<-1ex>[r]_U & \M^{\ob\R} \ar@<-1ex>[l]_L \ar@{}[l]|\perp \ar@{}[r]|\perp \ar@<-1ex>[r]_R & \M^{\R^-} \ar@<-1ex>[l]_V}\]
\end{lem}
\begin{proof}
The component of the distributive law $TK \Rightarrow KT$ at $\Phi \in \M^{\ob\R}$ is given at $r \in \ob\R$ by the map
\[ \xymatrix{\coprod\limits_{x,\R^+(x,r)} (\prod\limits_{y,\R^-(x,y)}\Phi(y)) \ar[rr]^{\chi_{\Phi,r}} \ar[dr]_{\mathrm{proj}_q} & & \prod\limits_{y,\R^-(r,y)}  (\coprod\limits_{x, \R^+(x,y)} \Phi(x)) \\ & \Phi(z) \ar[ur]_{\mathrm{incl}_i}} \] defined on the components indexed by $f \colon x \rightarrowtail r$ and $g \colon r \twoheadrightarrow y$ by forming the Reedy factorization $x \stackrel{q}{\twoheadrightarrow} z \stackrel{i}{\rightarrowtail} y$ of the composite $gf$, projecting to the component indexed by $q \in \R^-(x,z)$, and injecting via the component indexed by $i \in \R^+(z,y)$.
\end{proof}

By the definition of the latching and matching objects, the Reedy model structure on $\M^\R$ is left-induced from the right-induced model structure on $\M^{\R^+}$ and also right-induced from the left-induced model structure on $\M^{\R^-}$. { Lemmas \ref{lem:reedy-exact} and \ref{lem:reedy-bialgebra} combine to tell us that the Reedy model structure on $\M^\R$ can be seen as a special case of the model structures on bialgebras constructed in Corollary \ref{cor:monad_comonad}, a fact we record in the following proposition.

\begin{prop}\label{prop:reedy_m_s}
Reedy diagrams are bialgebras with respect to the monad on $\M^{\ob\R}$ induced by the direct subcategory $\R^+\subset \R$ and the comonad on $\M^{\ob\R}$ induced by the inverse subcategory $\R^-\subset\R$. The Reedy model structure on $\M^\R$ is left-induced from the projective model structure on $\M^{\R^+}$ and right-induced from the injective model structure on $\M^{\R^-}$.
\end{prop}
\begin{proof}
Recall that in the right-induced model structure on $\M^\R$, the restriction functor $U$ creates weak equivalences and fibrations, i.e.,  weak equivalences are objectwise, and fibrations are the maps that after restricting to $\M^{\R^-}$ are fibrations with respect to left-induced model structure on $\M^{\R^-}$. Since weak equivalences and cofibrations are objectwise in $\M^{\R^-}$, and ${\R^-}$ is a inverse category, the fibrations are the natural transformations $\psi$ such that all relative matching maps are fibrations in $\M$. The dual argument shows that the model structure on $\M^\R$ that is left-induced from the right-induced model structure on $\M^{\R^+}$ is also a Reedy model structure.
\end{proof}}

\subsection{Generalized Reedy diagrams as bialgebras}\label{ssec:generalized}

The results in the previous section can be extended to diagrams indexed by generalized Reedy categories, introduced in \cite{berger-moerdijk}. A (classical) Reedy category has no non-identity automorphisms; the idea of this generalization is to allow non-trivial automorphisms. 

\begin{defn} A \emph{generalized Reedy category} is a small category $\R$ together with a degree function $d: \ob \R \lra \mathbb{N}$ and two wide subcategories $\R^+$ and $\R^{-}$ satisfying the following axioms.
\begin{itemize}
\item Non-{invertible morphisms of $\R^+$ strictly} raise degree.
\item  Non-{invertible morphisms of $\R^{-}$ strictly} lower degree.
\item Isomorphisms in $\R$ preserve degree.
\item $\R^+ \bigcap \R^{-} = \Iso(\R)$, the subcategory of isomorphisms.
\item Every morphism in $\R$ factors uniquely up to isomorphism as a morphism in $\R^{-}$ followed by a morphism in $\R^+$.
\item If {$\theta$ is an isomorphism such that} $\theta f = f$ for  {all} $f \in \R^{-}$, then $\theta$ is an identity.
\end{itemize}
A \emph{dualizable} generalized Reedy category is a generalized Reedy category satisfying the following additional axiom.
\begin{itemize}
\item If {$\theta$ is an isomorphism such that} $\theta f = f$ for  {all} $f \in \R^{+}$, then $\theta$ is an identity.
\end{itemize}
\end{defn}

 If $\R$ is a dualizable generalized Reedy category, then $\R^\op$ is as well.  Most known examples of generalized Reedy categories are dualizable.

\begin{exs} Examples of dualizable generalized Reedy categories include all groupoids, the category of finite sets, orbit categories of finite groups, and the tree category $\Omega$, presheaves on which is the category of dendroidal sets.
\end{exs}

\begin{defn} Let $\R$ be a generalized Reedy category.  A model category $(\M, \Fib, \Cof, \WE)$ is  \emph{$\R$-projective}  if $\M^{\mathrm{Aut}(r)}$ admits the projective model structure for every object $r \in \R$. 
\end{defn}

If $(\M, \Fib, \Cof, \WE)$ is cofibrantly generated and permits the small object argument, then it is $\R$-projective for any generalized Reedy category $\R$ by \cite[11.6.1]{hirschhorn}.

\begin{thm}[{\cite[Theorem 1.6]{berger-moerdijk}}]\label{modelstructure} 
If $(\M, \Fib, \Cof, \WE)$ is an $\R$-projective model category, and $\R$ is a generalized Reedy category, then  $\M^{\R}$ admits a model structure where a map $\tau\colon  \Phi \lra \Psi$ is a
\begin{itemize}
\item weak equivalence if and only if for every $r \in \ob\R$, $\tau_r: \Phi(r) \lra \Psi(r)$ is a weak equivalence in the projective model structure on $\M^{\mathrm{Aut}(r)}$;
\item cofibration if and only if for every $r \in \ob\R$, the relative latching map $\ell_{r}(\tau)$ is a cofibration in the projective model structure on  $\M^{\mathrm{Aut}(r)}$;
\item fibration if and only if for every $r \in \ob\R$, the relative matching map $m_{r}(\tau)$ is a fibration in the projective model structure on   $\M^{\mathrm{Aut}(r)}$.
\end{itemize}
\end{thm}

We have, as before, an interesting description of a model structure on $\M^\R$ where $\R$ is a generalized Reedy category and $\M$ is locally presentable. 
We start with a diagram
\[ \xymatrix@R=4pc@C=4pc{\M^{\ob\R} \ar@<1ex>[r]^{L}  \ar@{}[r]|\perp & \M^{\Iso(\R)} \ar@{}[r]|{\perp} \ar@<1ex>[l]^U  \ar@<-1ex>[d]_L \ar@<-1ex>[r]_{R} & \M^{\R^-}  \ar@<1ex>[d]^L \ar@<-1ex>[l]_V \\  & \ar@{}[r]|{\top} \ar@{}[u]|{\dashv} \M^{\R^+} \ar@<1ex>[r]^{R} \ar@<-1ex>[u]_U & \M^{\R} \ar@<1ex>[u]^U \ar@<1ex>[l]^V\ar@{}[u]|{\vdash} }\] 
where  $U,V$ are restriction functors, and $L,R$ are left and right Kan extensions respectively. Here, $\Iso(\R)$ denotes the underlying groupoid of $\R$, i.e., a wide subcategory consisting of all isomorphisms. 

Equip $\M^{\Iso(R)}$ with the projective model structure, created by the forgetful functor to $\M^{\ob\R}$. {If $\M$ is an $\R$-projective model category, then  $\M^{\Iso(\R)}$ admits the projective model structure.}
 We will show that the restriction functors $U,V$ create the projective generalized Reedy model structures on $\M^\R$ in each case. 

{
\begin{lem}\label{lem:gen-reedy-exact} 
$LV \cong VL$ and $RU \cong UR$.
\end{lem}
\begin{proof}
Again, the assertions are equivalent, and dual. For $X \in \M^{\R^-}$ we have
\[  L V X(d) = \int^{x \in \Iso(\R)} \coprod_{\R^+(x,d)} X(x)\] 
and
\[ V L X(d) = \int^{x \in \R^-} \coprod_{\R(x,d)} X(x) = \mathrm{coeq} ( \coprod_{\R^-(y,x) \times \R(x,d)} X(y) \rightrightarrows \coprod_{\R(y,d)} X(y))\] As in the proof of Lemma \ref{lem:reedy-exact}, we see again these are equal by the fact that any morphism has a Reedy factorization (through an object of minimal degree) and any two such are connected by an isomorphism between the minimal degree objects. 
\end{proof}

The categories $\R^+$ and $\R^-$ are special cases of generalized Reedy categories with no degree-decreasing and no degree-increasing morphisms, respectively. As in the previous section, the projective generalized Reedy model structures on $\M^{\R^+}$ and $\M^{\R^-}$ are, respectively, right- and left-induced from the projective model structure on $M^{\Iso(\R)}$, and the projective generalized Reedy model structure on $\M^\R$ is simultaneously left-induced from the former and right-induced from the latter. As in Proposition \ref{prop:reedy_m_s}, we can understand generalized Reedy diagrams as ``bialgebras'' for the monad associated to $\R^+$ and the comonad associated to $\R^-$.}

There is a dual version of the model structure of Theorem \ref{modelstructure} that begins with the injective model structure on an automorphism category of any object in $\R$. In this case we can dualize the argument above, starting with the injective model structure on $\M^{\Iso(\R)}$. Note that the Reedy model structure is a special case of the generalized Reedy model structure. For an ordinary Reedy category $\R$, $\Iso(\R)$ is discrete and so the projective and injective model structures on $\M^{\Iso(\R)}$ reduce to the pointwise-defined model structure on $\M^{\ob(\R)}$.

\subsection{Reedy subcategories and exact squares}\label{ssec:exact}

{Formal facts about homotopy limits and colimits, computed in a model category or even in more general contexts, follow from the axioms of a \emph{derivator}, in the terminology of Grothendieck, called simply a \emph{homotopy theory} by Heller \cite{heller}. Central to this axiomatic framework is the notion of \emph{exact square}. Lemmas \ref{lem:reedy-exact} and \ref{lem:gen-reedy-exact} show that the square of inclusions of the direct and inverse subcategories is an exact square, a fact we record here for future reference.

\begin{defn}\label{def:exact} A diagram of small categories
\[ \xymatrix@R=3pc@C=3pc{ \A \ar[r]^f \ar[d]_u \ar@{}[dr]|{\Downarrow\alpha} & \B \ar[d]^v \\ \C \ar[r]_g & \D}\] is \emph{exact} if for any complete and cocomplete category $\C$ (either) mate of 
\[ \xymatrix@R=3pc@C=3pc{ \M^\A \ar@{}[dr]|{\Downarrow\alpha^*}  & \M^\B \ar[l]_{f^*} \\ \M^\C \ar[u]^{u^*} & \M^\D \ar[u]_{v^*} \ar[l]^{g^*}}\] defines an isomorphism $\lan_u f^* \Rightarrow g^*\lan_v$. 
\end{defn}

\begin{cor} If $(\R, \R^+,\R^-)$ is a Reedy category or generalized Reedy category, then the diagram of inclusions of indexing categories 
\[ \xymatrix@R=3pc@C=3pc{ \Iso(\R) \ar@{}[dr]|(.2){\lrcorner} \ar[r] \ar[d] & {\R^-} \ar[d] \\ {\R^+} \ar[r] & \R}\] 
is exact.
\end{cor}
\begin{proof}
This is the content of Lemmas \ref{lem:reedy-exact} and \ref{lem:gen-reedy-exact}.
\end{proof}
}

\section*{Addendum}
\renewcommand\thesection{A}
\setcounter{thm}{0}
\setcounter{subsection}{0}

We are grateful to Peter May for having suggested the following additional application of the theory of left-induced model structures.

\begin{prop} The adjunction \[ \xymatrix@C=4pc@R=4pc{ \sSet \ar@<1ex>[r]^{|-|} \ar@{}[r]|{\perp} & \cat{Top} \ar@<1ex>[l]^{\mathrm{Sing}} }\] left-induces the Quillen model structure on $\sSet$ from the Quillen model structure and the Hurewicz model structure on $\cat{Top}$.
\end{prop}
\begin{proof} Consider first the Quillen model structure on $\cat {Top}$, in which weak equivalences are weak homotopy equivalences, fibrations are Serre fibrations, and cofibrations are retracts of relative cell complexes.
By definition $|-|$ creates the weak homotopy equivalences in $\sSet$. Moreover the geometric realization of a simplicial monomorphisms is clearly a relative cell complex. If $|f|: |X| \lra |Y|$ is a retract of a relative cell complex, then it is a monomorphism and so is injective on interiors of cells, whence $f: X \lra Y$ is injective on non-degenerate simplices, carrying them to non-degenerate simplices. Suppose that $x \cdot \epsilon$ and $x' \cdot \epsilon'$ are degenerate simplices, given in their Eilenberg-Zilber decomposition.  In particular, $x$ and $x'$ are nondegenerate.  Since $f(x) \cdot \epsilon = f(x') \cdot \epsilon'$, and $f(x)$ and $f(x')$ are non-degenerate, the Eilenberg-Zilber lemma implies that  $\epsilon = \epsilon'$ and $f(x)=f(x')$, whence $x = x'$.  It follows that $f$ is a simplicial monomorphism if and only if its geometric realization is the retract of a relative cell complex.

Consider next the Hurewicz model structure on $\cat {Top}$, in which weak equivalences are homotopy equivalences, fibrations are Hurewicz fibrations, and cofibrations are closed cofibrations. It suffices to observe that $|X| \lra |Y|$ is a weak homotopy equivalence if and only if it is a homotopy equivalence. The argument in the previous paragraph implies that if $|f| \colon |X| \lra |Y|$ is a closed cofibration, and therefore, in particular, a monomorphism, then $f$ is a monomorphism.
\end{proof}

\end{document}